\documentclass[12pt]{amsart}%
\usepackage{amsfonts}
\usepackage{amsmath}
\usepackage{amssymb}
\usepackage{graphicx}
\usepackage[pagebackref]{hyperref}
\makeatletter
\@addtoreset{equation}{section}
\makeatother

\newtheorem{theorem}{Theorem}[section]

\newtheorem{lemma}[theorem]{Lemma}

\newtheorem{remark}[theorem]{Remark}

\begin{document}
\title[Optimal lower bound for  stability of Sobolev inequality]{Optimal asymptotic lower bound for stability of fractional Sobolev inequality and the  stability of Log-Sobolev inequality on the sphere}
\author{Lu Chen}
\address[Lu Chen]{Key Laboratory of Algebraic Lie Theory and Analysis of Ministry of Education, School of Mathematics and Statistics, Beijing Institute of Technology, Beijing
100081, PR China}
\email{chenlu5818804@163.com}

\author{Guozhen Lu}
\address[Guozhen Lu]{Department of Mathematics, University of Connecticut, Storrs, CT 06269, USA}
\email{guozhen.lu@uconn.edu}

\author{Hanli Tang}
\address[Hanli Tang]{Laboratory of Mathematics and Complex Systems (Ministry of Education), School of Mathematical Sciences, Beijing Normal University, Beijing, 100875, China}
\email{hltang@bnu.edu.cn}
\address{}
\thanks{The first author was partly supported by the National Key Research and Development Program (No.
2022YFA1006900) and National Natural Science Foundation of China (No. 12271027). The second author was partly supported by a Simons grant and a Simons Fellowship from the Simons Foundation.  The third author
was partly supported by National Key Research and Development Program (No. 2020YFA0712900) and the Fundamental Research Funds for the Central Universities(2233300008)}


\begin{abstract}
In this paper, we establish the optimal asymptotic lower bound for the stability of fractional Sobolev inequality:
\begin{equation}\label{Sob sta ine}
\left\|(-\Delta)^{s/2} U \right\|_2^2 - \mathcal S_{s,n} \| U\|_{\frac{2n}{n-2s}}^2\geq C_{n,s} d^{2}(U, \mathcal{M}_s),
\end{equation}
where $\mathcal{M}_s$ is the set of maximizers of the fractional Sobolev inequality of order $s$, $s\in (0, 1)$ and $C_{n,s}$ denotes the optimal lower bound of stability. We prove that the optimal lower bound $C_{n,s}$ behaves asymptotically at the order of $\frac{1}{n}$ when $n\rightarrow +\infty$ for any fixed $s\in (0,1)$. This  extends the work by Dolbeault-Esteban-Figalli-Frank-Loss \cite{DEFFL} on the stability of the first order Sobolev inequality and quantify the asymptotic behavior for lower bound of stability of fractional Sobolev inequality established by the current author's previous work in \cite{CLT1} in the case of $s\in (0, 1)$. Moreover, $C_{n,s}$ behaves asymptotically at the order of $s$ when $s\rightarrow 0$ for any given dimension $n$. (See Theorem \ref{sta of so} for these asymptotic estimates.)  As an application of this asymptotic estimate as $s\to 0$, we also derive the global stability for the log-Sobolev inequality on the sphere established by Beckner in \cite{Be1992,Be1997} with the optimal asymptotic lower bound on the sphere through the stability of fractional Sobolev inequalities with optimal asymptotic lower bound and the end-point differentiation method (see Theorem \ref{sta of lg}). This sharpens the earlier work by the authors in \cite{CLT} where only the local stability for the log-Sobolev inequality  on the sphere was proved. We also obtain the asymptotically optimal lower bound for the Hardy-Littlewood-Sobolev inequality when $s\to 0$ for fixed dimension $n$ and when $n\to \infty$ for fixed $s\in (0, 1)$ (See Theorem \ref{HLS} and the subsequent Remark \ref{Remark-HLS}).
\end{abstract}

\maketitle
\section{Introduction}
\vskip0.3cm
The sharp Sobolev inequality in $\mathbb{R}^n$  for $0<s<\frac {n}{2}$  states
\begin{equation}
	\label{eq-sob}
	\left\|(-\Delta)^{s/2} U \right\|_2^2 \geq \mathcal S_{s,n} \| U\|_{\frac{2n}{n-2s}}^2
	\qquad\text{for all}\ U\in \dot H^s(\mathbb{R}^n)
\end{equation}
with
\begin{equation}
	\label{eq:sobconst}
	\mathcal S_{s,n} = (4\pi)^s \ \frac{\Gamma(\frac{n+2s}{2})}{\Gamma(\frac{n-2s}{2})} \left( \frac{\Gamma(\frac n2)}{\Gamma(n)} \right)^{2s/n}
	= \frac{\Gamma(\frac{n+2s}{2})}{\Gamma(\frac{n-2s}{2})} \ |\mathbb{S}^n|^{2s/n} \,,
\end{equation}
where $\dot H^s(\mathbb{R}^n)$ denotes the completion of $C_{c}^{\infty}(\mathbb{R}^n)$ under the norm $\|(-\Delta)^{\frac{s}{2}}U\|_2$. The sharp constant $\mathcal S_{s,n}$ of inequality (\ref{eq-sob}) has been computed first by Rosen \cite{Ro} in the case $s=1$, $n=3$ and then independently by Aubin \cite{Au} and Talenti \cite{Ta} in the case $s=1$ and general $n$. For general $s$, the sharp inequality was proved by Lieb in \cite{Li} as an equivalent reformulation of the sharp Hardy-Littlewood-Sobolev inequality in the case of conformal index. Furthermore, he also showed that the equality of Sobolev inequality \eqref{eq-sob} holds if and only if
$$U\in \mathcal{M}_s:=\left\{cV(\frac{\cdot-x_0}{a}): c\in \mathbb{R},\  x_0\in \mathbb{R}^n,\  a>0 \right\},$$
where $V(x)=(1+|x|^2)^{-\frac{n-2s}{2}}$.

\medskip

By the stereographic projection, we know that $\mathbb{R}^n$ (or rather $\mathbb{R}^n\cup\{\infty\}$) and $\mathbb{S}^n$ ($\subset \mathbb{R}^{n+1}$) are conformally equivalent. Thus, there exists an equivalent version of \eqref{eq-sob} on $\mathbb{S}^n$. This form was found explicitly by Beckner in \cite[Eq.~(19)]{Be1993}, namely,
\begin{eqnarray}
	\label{eq:sobsphere}
	\left\| A_{2s}^{1/2} u \right\|_2^2 \geq \mathcal S_{s,n} \|u\|_{\frac{2n}{n-2s}}^2
	\qquad\text{for all}\ u\in H^s(\mathbb{S}^n)
\end{eqnarray}
with
\begin{align}
	\label{eq:opas}
	A_{2s} = \frac{\Gamma(B+\tfrac12 + s)}{\Gamma(B+\tfrac12 - s)}
	\qquad\text{and}\qquad
	B = \sqrt{-\Delta_{\mathbb{S}^n} + \tfrac{(n-1)^2}{4}}, \
\end{align}
where $-\Delta_{\mathbb{S}^n}$ denotes the Laplace-Beltrami operator on the sphere $\mathbb{S}^n$ and the operator $B$ and $A_{2s}$ acting on spherical harmonics $Y_{l,m}$ of degree $l$   satisfy
 $$B Y_{l,m}=(l+\frac{n-1}{2})Y_{l,m},~~~A_{2s}Y_{l,m}=\frac{\Gamma(l+n/2+s)}{\Gamma(l+n/2-s)}Y_{l,m}.$$
 (We refer the reader to the books \cite{M} and \cite{SteinWeiss} for detailed exposition of spherical harmonics.)
In fact, the operators $A_{2s}$ is a $2s$-order conformally invariant differential operator and can be written as  $(-\Delta_{\mathbb{S}^n})^{s}+lower\ order\  terms$. For the integer $s$, they are related to the GJMS operators in conformal geometry \cite{FG, GrJeMaSp}.
Beckner also proved that the equality holds in (\ref{eq:sobsphere}) if and only if
\begin{eqnarray}\label{opt set}
u\in M_s:=\left\{c\left(\frac{\sqrt{1-|\xi|^2}}{1-\xi\cdot \omega}\right)^{\frac{n-2s}{2}}:\  \xi\in B^{n+1}, \ c\in \mathbb{R} \right\},
\end{eqnarray}
where $B^{n+1}:=\{\xi\in \mathbb{R}^{n+1}, |\xi|<1\}$.

Inequality (\ref{eq:sobsphere}) becomes an equality as $s\rightarrow 0$. Differentiating at $s=0$ and using the Funk-Hecke formula, Beckner \cite{Be1992,Be1997} proved the following conformally invariant logarithmic Sobolev inequality on
$\mathbb{S}^n$.
\vskip0.5cm
\textbf{Theorem A. }
\textit{Assume} $u\in L^2(\mathbb{S}^n)$. \textit{Then}
\begin{align}\label{L-S ine}
\iint_{\mathbb{S}^n \times \mathbb{S}^n} \frac{|u(\omega) - u(\eta)|^2}{|\omega - \eta|^n} \, d\omega d\eta  \geq C_n \int_{\mathbb{S}^n} |u(\omega)|^2 \ln \frac{|u(\omega)|^2 |\mathbb{S}^n|}{\|u\|_2^2} d\omega
\end{align}
 \textit{with sharp constant}
 \begin{equation}
\label{eq:becknerconst}
C_n = \frac{4}{n} \frac{\pi^{n/2}}{\Gamma(n/2)} \,.
\end{equation}
\textit{The equality holds if and only if}
  \begin{eqnarray}
\label{eq:beckneropt}
u(\omega) = c \left( 1-\xi\cdot\omega\right)^{-n/2}
\end{eqnarray}
 \textit{for some} $\xi\in B^{n+1}$ \textit{and some} $c\in \mathbb{R}$.
 \vskip0.1cm

We also note that Frank, K\"{o}nig and the third author \cite{FKT} classified that all the positive solutions of the associated Euler-Lagrange equation
 to the log-Sobolev inequality on the sphere are of the form \eqref{eq:beckneropt}.
\vskip0.5cm

In the past decades, there have been extensive works done in studying the
stability of functional and geometric inequalities. The stability of Sobolev inequality started from Brezis and Lieb. In \cite{BrLi} they asked if the following refined first order Sobolev inequality ($s=1$ in (\ref{eq-sob}))
holds for some distance function $d$:
$$\left\|(-\Delta)^{1/2} U \right\|_2^2 - \mathcal S_{1,n} \| U\|_{\frac{2n}{n-2}}^2\geq c d^{2}(U, \mathcal{M}_1).$$
This question was answered affirmatively  in a pioneering work by Bianchi and Egnell \cite{BiEg}, in the case $s=2$ by the second author and Wei \cite{LuWe} and in the case of any positive even integer $s<n/2$ by
Bartsch, Weth and Willem \cite{BaWeWi}. In 2013, Chen, Frank and Weth \cite{ChFrWe} established the stability of Sobolev inequality for all $0<s<n/2$. They proved that
\begin{equation}\label{Sob sta ine}
\left\|(-\Delta)^{s/2} U \right\|_2^2 - \mathcal S_{s,n} \| U\|_{\frac{2n}{n-2s}}^2\geq C_{n,s} d^{2}(U, \mathcal{M}_s),
\end{equation}
for all $U\in \dot H^s(\mathbb{R}^n)$, where $C_{n,s}$ denotes the optimal lower bound of stability of fractional Sobolev inequality, $d(U,\mathcal{M}_s)=\min\{\|(-\Delta)^{s/2}(U-\phi)\|_{L^2}:\phi \in \mathcal{M}_s\}$.  Recently, Dolbeault, Esteban, Figalli, Frank and Loss in \cite{DEFFL} obtained for the first time the optimal asymptotic behavior for the lower bound $C_{n,s}$ when $s=1$. In fact, they proved

\vskip0.3cm
\textbf{Theorem B. }
\textit{There is an explicit constant} $\beta>0$ \textit{such that for all} $n\geq 3$ \textit{and for all} $U\in \dot H^1(\mathbb{R}^n)$, \textit{there holds}
$$\left\|\nabla U \right\|_2^2 - \mathcal S_{1,n} \| U\|_{\frac{2n}{n-2}}^2\geq \frac{\beta}{n}d^2(U,\mathcal{M}_1).$$
As an application, they also established the stability for Gaussian log-Sobolev inequality.  The current authors  gave in \cite{CLT1}  the explicit lower bounds of general fractional Sobolev inequalities of order s when $0<s<\frac{n}{2}$ (including high-order and fractional-order cases) and Hardy-Littlewood-Sobolev inequalities.  However, it should be noted that the explicit lower bound obtained in \cite{CLT1} is not optimal in the asymptotic sense when $n\rightarrow +\infty$ for fractional order $0<s<\frac{n}{2}$. In this paper, we derive two main results: the optimal asymptotic lower bound when $n\to \infty$ for any fixed $s\in (0, 1)$ and when $s\to 0$ for any fixed dimension $n$. The latter result allows us to establish the global stability of the log-Sobolev inequality on the sphere using the endpoint differentiation theorem. We also note that the current authors in \cite{CLT2} have also obtained the stability for the high-order Sobolev and  fractional Sobolev inequalities for $1\leq s<\frac{n}{2}$ as $n\to \infty$ with the dimension-dependent constants. Moreover, we should point out that the method used in \cite{CLT2} is substantially different from the one used in this paper here and  only works for $1\leq s<\frac{n}{2}$, but not for $0<s<1$. Thus, this paper and \cite{CLT2} complement to each other.

\vskip0.1cm

For the study of stability of other kinds of functional
and geometric inequalities, we refer the interested readers to the papers \cite{BDNS}, \cite{CaF}, \cite{CFLL},
\cite{CFMP}, \cite{DFLL}, \cite{FiNe}, \cite{FiZh}, \cite{JF1}, \cite{JF2}, \cite{Ko1}, \cite{WW}  and references therein. The first main purpose of this paper is to establish the optimal asymptotic lower bound for stability of fractional Sobolev inequality when $n\rightarrow +\infty$ for any fixed $s\in (0, 1)$ and $s\rightarrow 0$ for any fixed dimension $n$ respectively. Our first result states

\begin{theorem}\label{sta of so}
For $0<s<1$ and $n\geq 3$ or $0<s<\frac{1}{6}$ and $n\geq 1$, then there exists a bounded constant $\beta_{s,n}$ such that for any $f\in H^{s}(\mathbb{S}^n)$, there holds
$$\left\| A_{2s}^{1/2} u \right\|_2^2-\mathcal S_{s,n} \|u\|_{\frac{2n}{n-2s}}^2\geq \frac{s}{n}\beta_{s,n} \inf_{h\in M_s}\|A_{2s}^{1/2}(f-h)\|\|_2^2.$$
\end{theorem}
\begin{remark}
 We only prove that Theorem \ref{sta of so} holds when $0<s<\min\{1,n/6\}$. It includes the case $0<s<\frac{1}{6}$ $(n\geq 1)$
and the case $0<s<1$, ($n\geq 6$). When $s\in (0,1)$ and $n=3$, $4$, $5$,  we can apply the lower bound estimate for stability of general fractional Sobolev inequality obtained in \cite{CLT1} to get the lower bound $\beta_{s,n}$ independent of $s$ and $n$. We also note that the lower bound estimate obtained in \cite{CLT1} is away from zero when $n$ is away from $2s$. The explicit form of $\beta_{s,n}$ for $0<s<\min\{1,n/6\}$ can be found in (\ref{beta}) and in fact $\lim\limits_{(s,n)\rightarrow (0, +\infty)}\beta_{s,n}=\frac{\delta_0\varepsilon_0}{4(1+\delta_0)}$, where $\varepsilon_0\in (0,1/12)$ is a constant and $\delta_0$ is an explicit constant given by (\ref{delta0}) independent of $n$ and $s$. From Chen, Frank and Weth's work in \cite{ChFrWe}, we know
that the optimal lower bound $C_{n,s}$ satisfies
$$C_{n,s}\leq \frac{4s}{n+2+2s}.$$ Later, K$\ddot{o}$nig in \cite{Ko} proved that the optimal lower bound is strictly smaller than $\frac{4s}{n+2+4s}$. So the lower bound for stability of fractional Sobolev inequality we obtain in Theorem~\ref{sta of so} is optimal with $s$ and $n$ when $s\rightarrow 0$ and $n\rightarrow +\infty$ respectively.
 \end{remark}
\begin{remark}
Due to their equivalence between the fractional Sobolev inequalities on $\mathbb{R}^n$ and the sphere $\mathbb{S}^n$, Theorem \ref{sta of so}
also gives the same result of asymptotic estimates for the stability of fractional Sobolev inequalities \eqref{Sob sta ine} on $\mathbb{R}^n$.
\end{remark}

From the proof of Theorem \ref{sta of so}, we in fact also derive the asymptotic lower bound for the stability of the Hardy-Littlewood-Sobolev inequality when $s\rightarrow 0$ and $n\rightarrow +\infty$ respectively (see Section 3, Step 3-Step 4).
\begin{theorem}\label{HLS}
For $0<s<1$ and $n\geq 3$ or $0<s<\frac{1}{6}$ and $n\geq 1$, then there exists an constant $\beta$ independent of $s$ and $n$ such that for any $g\in L^{\frac{2n}{n+2s}}(\mathbb{R}^n)$, there holds
\begin{align}\label{sta hls func}
\|g\|^2_{\frac{2n}{n+2s}}-\mathcal S_{s,n} \|(-\Delta)^{-s/2}g\|_2^2\geq \frac{s}{n}\beta \inf\limits_{h\in M_{HLS}}\|g-h\|^2_{\frac{2n}{n+2s}}.
\end{align}
\end{theorem}

\begin{remark}\label{Remark-HLS}
We would like to point out that the lower bound for the stability of the Hardy-Littlewood-Sobolev  inequality in Theorem \ref{HLS} is  asymptotically  optimal when $s\rightarrow 0$ for fixed dimension $n$ and when $n\rightarrow +\infty$ for fixed $s\in (0, 1)$. In fact, if we denote by $C_{HLS}$ and $C_{S}$ the optimal lower bounds for the stability of HLS inequality and Sobolev inequality respectively, then by duality we can prove $C_{S}\geq  \frac{1}{2}C_{HLS}$ by repeating the proof of Section 4 in [15]. Since  $C_{S}\leq \frac{4s}{n+2+2s}<1$ when $0<s<\frac{n}{2}$, we must have $C_{HLS}\leq \frac{8s}{n+2+2s}$.
\end{remark}

\vskip0.3cm
The second goal of this paper is to establish the global stability of log-Sobolev inequality on the sphere. Such a log-Sobolev inequality on the sphere was established by Beckner in \cite{Be1992,Be1997}. However, the stability of this log-Sobolev inequality remains unsolved.  Recall that the authors in \cite{CLT} have established the local stability of the log-Sobolev inequality on the sphere.
Let
$$M=\left\{v_{c,\xi}(\omega)=c\left(1-\xi\cdot w\right)^{-n/2}:\, ~\xi\in B^{n+1}, \ c\in \mathbb{R}\right\}$$
be the optimizer set of the log-Sobolev inequality on the sphere $\mathbb{S}^n$. Denote the log-Sobolev functional by
$$LS(u)=\iint_{\mathbb{S}^n \times \mathbb{S}^n} \frac{|u(\omega) - u(\eta)|^2}{|\omega - \eta|^n} \, d\omega d\eta  - C_n \int_{\mathbb{S}^n} |u(\omega)|^2 \ln \frac{|u(\omega)|^2 |\mathbb{S}^n|}{\|u\|_2^2} d\omega.$$
We also denote $$\mathcal{D_L}=\left\{v\in L^2(\mathbb{S}^n):\iint\limits_{\mathbb{S}^n\times \mathbb{S}^n}\frac{|v(\omega)-v(\eta)|^2}{|\omega-\eta|^n}d \omega d \eta<\infty\right\},$$
and
$$d^2(u,M)=\inf\{\|u-\phi\|^2_{L^2(\mathbb{S}^n)}:\phi\in M\}=\inf_{(c,\xi)\in \mathbb{R}\times B^{n+1}}\left\langle u-v_{c,\xi}, u-v_{c,\xi} \right\rangle,$$
where $B^{n+1}=\{x\in \mathbb{R}^{n+1}: \,|x|<1\}$ and $d\omega$ is the surface measure on the sphere with $\int_{\mathbb{S}^n}d\omega=|\mathbb{S}^n|$.
The authors  proved in ~\cite{CLT} that

\vskip0.5cm
\textbf{Theorem C. }
\textit{Let} $u\in \mathcal{D_L}$ \textit{with} $\|u\|_{L^2}=1$. \textit{Then}
$$LS(u) \geq \frac{8\pi^{n/2}}{\Gamma(n/2)(n+2)} d^2(u,M)+o(d^2(u,M)),$$
\textit{where} $o(d^2(u,M))$ \textit{is only dependent on} $d(u,M)$ \textit{but independent of} $u$. \textit{Moreover, the constant}
$\frac{8\pi^{n/2}}{\Gamma(n/2)(n+2)}$ \textit{is sharp}.
\vskip0.5cm

The main difficulty to derive the global stability of the log-Sobolev inequality on the sphere from the local one is that it seems difficult  to use the Lions' type of  concentration-compactness argument \cite{Lions1, Lions2} in the log-Sobolev setting, which is a classic method to obtain the global stability before the work of Dolbeault et al.~\cite{DEFFL}. Since the log-Sobolev inequality on the sphere can be viewed as the limiting case of the fractional Sobolev inequality on the sphere by letting $s\rightarrow 0$, then the same property should be true for the stability of log-Sobolev inequality. Since we have obtained the optimal asymptotic lower bound for stability of fractional Sobolev inequality when $s\rightarrow 0$ in Theorem \ref{sta of so}, this makes it possible to establish the global stability of log-Sobolev inequality via the end-point differentiation argument. Indeed, we derive the following global stability for the log-Sooblev inequality on the sphere.
\begin{theorem}\label{sta of lg}
There is an explicit constant $\alpha_n>0$ such that for all $n\geq 1$, $u\in \mathcal{D_L}$, there holds
$$LS(u) \geq \alpha_n d^2(u,M).$$
\end{theorem}
\vskip0.5cm
\begin{remark}
In fact from the proof of  Theorem \ref{sta of lg}, we can choose $\alpha_n=\frac{\pi^{\frac{n}{2}}\delta_0\varepsilon_0}{2n\Gamma(\frac{n}{2})(1+\delta_0)}$, where $\varepsilon_0\in (0,1/12)$ is a constant and $\delta_0$ is dependent only on $\varepsilon_0$ by (\ref{delta0}). On the other hand, through Theorem B, we see that $$\inf_{u\in \mathcal{D_L}}\frac{LS(u)}{d^2(u,M)}\leq \frac{8\pi^{n/2}}{\Gamma(n/2)(n+2)}.$$ Hence, we deduce that the lower bound for the stability of log-Sobolev inequality obtained in Theorem \ref{sta of lg} is optimal when $n\rightarrow +\infty$.
\end{remark}

This paper is organized as follows. Section 2 is devoted to a local version of stability for fractional Sobolev inequalities with the optimal asymptotic lower bound. In Section 3, we will give the stability of fractional Sobolev inequalities with the optimal asymptotic lower bound. In Section 4, we will establish the stability of the log-Sobolev inequality with the explicit lower bound. In Section 5, we add the proof of an auxiliary lemma there.
\medskip

{\bf Acknowledgement.} The authors wish to thank R. Frank for his comments on our earlier version  which have helped us to improve the clarity of the exposition of the paper.

\section{A local version of stability for fractional Sobolev inequalities with the optimal asymptotic lower bound}
In this section we will set up a local stability of fractional Sobolev inequalites with the optimal asymptotic lower bound when $0<s<\min\{1,n/6\}$ for positive functions. The idea of the proof essentially comes from \cite{DEFFL}, where Dolbeault,
Esteban, Figalli, Frank and Loss obtained the optimal asymptotic lower bound for the local stability of first-order Sobolev inequality. We first prove the following lemma.
\begin{lemma}\label{local stability}
For any $0<\varepsilon_0<1/12$, there exist an explicit $\delta\in(0,1)$ such that for all $0<s<\min\{1,n/6\}$ and all $0 \leq u\in H^{s}(\mathbb{S}^n)$ with
\begin{align}\label{distance}
\inf_{h\in M_s}\|A_{2s}^{1/2}(u-h)\|_2^2\leq \delta\|A_{2s}^{1/2}u\|_2^2,
\end{align}
there holds
$$\left\| A_{2s}^{1/2} u \right\|_2^2-\mathcal S_{s,n} \|u\|_{\frac{2n}{n-2s}}^2\geq \frac{4s}{n-2s} \varepsilon_0\inf_{h\in M_s}\|A_{2s}^{1/2}(u-h)\|_2^2.$$
\end{lemma}
We need to point out here that the constant $\delta=\frac{\delta_0}{1+\delta_0}$, and $\delta_0$ comes from Lemma~\ref{local sta}. Instead of directly proving Lemma~\ref{local stability}, we will prove Lemma~\ref{local sta}. In fact, let $0\leq u\in H^{s}(\mathbb{S}^n)$ with
$$\inf_{h\in M_S}\|A_{2s}^{1/2}(u-h)\|_2^2\leq \delta\|A_{2s}^{1/2}u\|_2^2.$$
It is well known (see~\cite{ChFrWe}) that the infimum is attained. By conformal and scaling invariance of the stability of Sobolev inequality, we may assume $u=1+r\geq 0$, and $r$ is orthogonal with spherical harmonics with degree $0$ and $1$.
Then (\ref{distance}) is equivalent to
$$\|A_{2s}^{1/2}r\|_2^2\leq \frac{\delta_0}{1+\delta_0}\|A_{2s}^{1/2}(1+r)\|_2^2= \frac{\delta_0}{1+\delta_0}\left(\frac{\Gamma(n/2+s)}{\Gamma(n/2-s)}|\mathbb{S}^n|+\|A_{2s}^{1/2}r\|_2^2\right).$$
That is
$$\|A_{2s}^{1/2}r\|_2^2\leq \frac{\Gamma(n/2+s)}{\Gamma(n/2-s)}|\mathbb{S}^n|\delta_0,$$
which implies
$$\|r\|_{\frac{2n}{n-2s}}^2\leq |\mathbb{S}^n|^{1-\frac{2s}{n}}\delta_0$$
by the Sobolev inequality.
\vskip0.1cm

On the other hand, from \cite{Be2015} (see Lemma10 and Theorem 11 in \cite{Be2015}), we know that
$$\|A_{2s}^{1/2}u\|_2^2=|\mathbb{S}^n|\frac{\Gamma(n/2+s)}{\Gamma(n/2-s)}\left(B_{n,s}\int_{\mathbb{S}^n}\int_{\mathbb{S}^n}\frac{|u(\xi)-u(\eta)|^2}{|\xi-\eta|^{n+2s}}d\sigma_\xi d\sigma_\eta+\int_{\mathbb{S}^n}|u|^2d\sigma_\xi\right),$$
and
$$\mathcal{S}_{n,s}\|u\|_{\frac{2n}{n-2s}}^2=|\mathbb{S}^n|\frac{\Gamma(n/2+s)}{\Gamma(n/2-s)}\left(\int_{\mathbb{S}^n}|u|^{\frac{2n}{n-2s}}d\sigma_\xi\right)^{\frac{n-2s}{n}},$$
where $d\sigma_\xi$ denote the uniform probability measure on $\mathbb{S}^n$ and $B_{n,s}=2^{n+2s}\frac{s\Gamma(\frac{n}{2}+1)\Gamma(\frac{n}{2}-s)}{\Gamma(n+1)\Gamma(1-s)}$.
 So the following Lemma will imply Lemma \ref{local stability}.

\begin{lemma}\label{local sta}
For any $0<\varepsilon_0<1/12$, there exist a $\delta_0\in(0,1)$ such that for all $0<s<\min\{1,n/6\}$ and all $-1\leq r\in H^{s}(\mathbb{S}^n)$ satisfying
$$\left(\int_{\mathbb{S}^n}|r|^{\frac{2n}{n-2s}}d\sigma_\xi\right)^{\frac{n-2s}{n}}\leq \delta_0~~~\textit{and}~~~\int_{\mathbb{S}^n}r d\sigma_\xi=0=\int_{\mathbb{S}^n}w_jrd\sigma_\xi, ~~j=1,\cdots,n+1.$$
There holds
\begin{align*}
&B_{n,s}\int_{\mathbb{S}^n}\int_{\mathbb{S}^n}\frac{|r(\xi)-r(\eta)|^2}{|\xi-\eta|^{n+2s}}d\sigma_\xi d\sigma_\eta+ \int_{\mathbb{S}^n}|1+r|^2d\sigma_\xi-\left(\int_{\mathbb{S}^n}(1+r)^{\frac{2n}{n-2s}}\,d\sigma_\xi\right)^{\frac{n-2s}{n}}\\
&\ \ \geq \frac{4s}{n-2s} \varepsilon_0\big(B_{n,s}\int_{\mathbb{S}^n}\int_{\mathbb{S}^n}\frac{|r(\xi)-r(\eta)|^2}{|\xi-\eta|^{n+2s}}d\sigma_\xi
 d\sigma_\eta+\int_{\mathbb{S}^n}r^2d\sigma_\xi\big) .
\end{align*}
\end{lemma}
\begin{remark}
It is important to point out that the constant $\delta_0$ is independent of $s$ and $n$, see (\ref{delta0}). And the reason why we choose to prove the local stability of this form is that the double integral will bring us benefit when we split the function $r$ and do the orthogonal analyis (see Lemma \ref{lem split}).
\end{remark}

Now, let us prove Lemma \ref{local sta}. For any $r\geq -1$, define $r_1$, $r_2$ and $r_3$ by
$$r_1=\min\{r, \gamma\},\ \ r_2=\min\{(r-\gamma)_{+}, M-\gamma\},\ \ r_3=(r-M)_{+},$$
where $\gamma$ and $M$ are two parameters such that $0<\gamma<M$. We need the following lemma.
\begin{lemma}\label{q-estimate}(\cite{DEFFL})
Given $\varepsilon>0$, $M>0$, and $\gamma\in (0, \frac{M}{2})$. There exists a constant $C_{\gamma,\varepsilon,M}$ such that for any
$r\geq -1$ and $2\leq q\leq 3$ and $\theta=q-2$, there holds
\begin{equation}\begin{split}
(1+r)^q-1-qr&\leq (\frac{1}{2}q(q-1)+2\gamma\theta)r_1^2+(\frac{1}{2}q(q-1)+C_{\gamma,\varepsilon,M}\theta)r_2^2\\
&\ \ +2r_1r_2+2(r_1+r_2)r_2+(1+\varepsilon\theta)r_3^q.
\end{split}\end{equation}
\end{lemma}

\subsection{Dividing the deficit}
First we split the double integral part of the deficit in the following lemma and the proof of the lemma will be proved in Section 5.
\begin{lemma}\label{lem split}
There holds
\begin{equation}\begin{split}
&\int_{\mathbb{S}^n}\int_{\mathbb{S}^n}\frac{|r(\xi)-r(\eta)|^2}{|\xi-\eta|^{n+2s}}d\sigma_\xi d\sigma_\eta
\geq \sum_{i=1}^{3}\int_{\mathbb{S}^n}\int_{\mathbb{S}^n}\frac{|r_i(\xi)-r_i(\eta)|^2}{|\xi-\eta|^{n+2s}}d\sigma_\xi d\sigma_\eta.
\end{split}\end{equation}
\end{lemma}

Then let us handle the the terms $\int_{\mathbb{S}^{n}}(1+r)^{2}d\sigma_{\xi}$ and $\left(\int_{\mathbb{S}^{n}}(1+r)^{q}\,d\sigma_{\xi}\right)^{2/q}$. Since $r$ has mean zero,
direct computation gives that
\begin{align}\label{est of L2}\nonumber
 \int_{\mathbb{S}^{n}}(1+r)^{2}d\sigma_{\xi}&=1 + \int_{\mathbb{S}^{n}}r^{2}d\sigma_{\xi}\\\nonumber
&= 1+\int_{\mathbb{S}^{n}}r_1^{2}\,d\sigma_{\xi} + \int_{\mathbb{S}^{n}}r_2^{2}\,d\sigma_{\xi} + \int_{\mathbb{S}^{n}}r_3^{2}\,d\sigma_{\xi}\\
&\ \  + 2\int_{\mathbb{S}^{n}}r_1r_2\,d\sigma_{\xi} + 2\int_{\mathbb{S}^{n}}(r_1+r_2)r_3\,d\sigma_{\xi}.
\end{align}

Denote $q=\frac{2n}{n-2s}$ and $\theta=q-2=\frac{4s}{n-2s}$, then $2\leq q\leq 3$ and $0<\theta<1$ since $n>6s$. Given two parameters $\varepsilon_1, \varepsilon_2>0$, we apply Lemma \ref{q-estimate} with (an arbitrary choice of $M\geq 2\gamma$)
$$\gamma=\frac{\varepsilon_1}{2}, \varepsilon=\varepsilon_2,\ \ C_{\gamma,\varepsilon,M}=C_{\varepsilon_1, \varepsilon_2},$$
$(1+x)^{\frac{2}{q}}\leq 1+\frac{2}{q}x$ and $\int_{\mathbb{S}^n}r d\sigma_\xi=0$
to derive that
\begin{align}\label{est of Lq}\nonumber
&\left(\int_{\mathbb{S}^{n}}(1+r)^{q}\,d\sigma_{\xi}\right)^{2/q}\\\nonumber
 &\ \ \leq 1 + (q-1+\frac{2}{q}\epsilon_1\theta)\int_{\mathbb{S}^{n}}r_1^{2}\,d\sigma_{\xi} + (q-1+\frac{2}{q}C_{\epsilon_1,\epsilon_2}\theta)\int_{\mathbb{S}^{n}}r_2^{2}\,d\sigma_{\xi}  \\\nonumber
 &\ \ \ \ + \frac{4}{q}\int_{\mathbb{S}^{n}}(r_1r_2)\,d\sigma_{\xi} + \frac{4}{q}\int_{\mathbb{S}^{n}}(r_1+r_2)r_3\,d\sigma_{\xi} + \frac{2}{q}(1+\epsilon_2\theta)\int_{\mathbb{S}^{n}}r_3^{q}\,d\sigma_{\xi}
\\&\ \ \leq 1 + (q-1+\epsilon_1\theta)\int_{\mathbb{S}^{n}}r_1^{2}\,d\sigma_{\xi} + (q-1+C_{\epsilon_1,\epsilon_2}\theta)\int_{\mathbb{S}^{n}}r_2^{2}\,d\sigma_{\xi}  \\\nonumber
&\ \ \ \ \ + 2\int_{\mathbb{S}^{n}}r_1r_2\,d\sigma_{\xi} + 2\int_{\mathbb{S}^{n}}(r_1+r_2)r_3\,d\sigma_{\xi} + \frac{2}{q}(1+\epsilon_2\theta)\int_{\mathbb{S}^{n}}r_3^{q}\,d\sigma_{\xi}.
\end{align}
Combining Lemma \ref{lem split}, (\ref{est of L2}) and (\ref{est of Lq}), we can divide the deficit as follows,
\begin{footnotesize}
\begin{align*}
&B_{n,s}\int_{\mathbb{S}^n}\int_{\mathbb{S}^n}\frac{|r(\xi)-r(\eta)|^2}{|\xi-\eta|^{n+2s}}d\sigma_\xi d\sigma_\eta+ \int_{\mathbb{S}^n}|1+r|^2d\sigma_\xi-(\int_{\mathbb{S}^{n}}(1+r)^{q}\,d\sigma_\xi)^{2/q}\\
&\ \ \geq \theta \varepsilon_0\left(B_{n,s}\int_{\mathbb{S}^n}\int_{\mathbb{S}^n}\frac{|r(\xi)-r(\eta)|^2}{|\xi-\eta|^{n+2s}}d\sigma_\xi
 d\sigma_\eta+\int_{\mathbb{S}^n}r^2d\sigma_\xi\right)\\
 &\ \ \ \ +\left(1-\theta\varepsilon_0)\big(B_{n,s}\int_{\mathbb{S}^n}\int_{\mathbb{S}^n}\frac{|r_1(\xi)-r_1(\eta)|^2}{|\xi-\eta|^{n+2s}}d\sigma_\xi
 d\sigma_\eta+\int_{\mathbb{S}^n}r_1^2d\sigma_\xi\right)+(1-q-\varepsilon_1\theta)\int_{\mathbb{S}^n}r_1^2d\sigma_\xi\\
 &\ \ \ \  +(1-\theta\varepsilon_0)\big(B_{n,s}\int_{\mathbb{S}^n}\int_{\mathbb{S}^n}\frac{|r_2(\xi)-r_2(\eta)|^2}{|\xi-\eta|^{n+2s}}d\sigma_\xi
 d\sigma_\eta+\int_{\mathbb{S}^n}r_2^2d\sigma_\xi\big)+(1-q-C_{\varepsilon_1,\varepsilon_2}\theta)\int_{\mathbb{S}^n}r_2^2d\sigma_\xi\\
&\ \ \ \ +(1-\theta\varepsilon_0)\big(B_{n,s}\int_{\mathbb{S}^n}\int_{\mathbb{S}^n}\frac{|r_3(\xi)-r_3(\eta)|^2}{|\xi-\eta|^{n+2s}}d\sigma_\xi
 d\sigma_\eta+\int_{\mathbb{S}^n}r_3^2d\sigma_\xi\big)-\frac{2}{q}(1+\varepsilon_2\theta)\int_{\mathbb{S}^n}r_3^qd\sigma_\xi.\\
\end{align*}
\end{footnotesize}
Let us define
\begin{align*}
I_1:&=(1-\theta\varepsilon_0)\big(B_{n,s}\int_{\mathbb{S}^n}\int_{\mathbb{S}^n}\frac{|r_1(\xi)-r_1(\eta)|^2}{|\xi-\eta|^{n+2s}}d\sigma_\xi
 d\sigma_\eta+\int_{\mathbb{S}^n}r_1^2d\sigma_\xi\big)\\
 &\ \ +(1-q-\varepsilon_1\theta)\int_{\mathbb{S}^n}r_1^2d\sigma_\xi +\sigma_0\theta\int_{\mathbb{S}^{n}}(r_2^{2}+r_3^{2})d\sigma_\xi,
\\I_2&:=(1-\theta\varepsilon_0)\big(B_{n,s}\int_{\mathbb{S}^n}\int_{\mathbb{S}^n}\frac{|r_2(\xi)-r_2(\eta)|^2}{|\xi-\eta|^{n+2s}}d\sigma_\xi
 d\sigma_\eta+\int_{\mathbb{S}^n}r_2^2d\sigma_\xi\big)\\
&\ \ +(1-q-C_{\varepsilon_1, \varepsilon_2}\theta-\sigma_0\theta)\int_{\mathbb{S}^n}r_2^2d\sigma_\xi
\\I_3&:=(1-\theta\varepsilon_0)\big(B_{n,s}\int_{\mathbb{S}^n}\int_{\mathbb{S}^n}\frac{|r_3(\xi)-r_3(\eta)|^2}{|\xi-\eta|^{n+2s}}d\sigma_\xi,
 d\sigma_\eta+\int_{\mathbb{S}^n}r_3^2d\sigma_\xi\big)\\
 &\ \ -\frac{2}{q}(1+\varepsilon_2\theta)\int_{\mathbb{S}^n}r_3^qd\sigma_\xi-\sigma_0\theta\int_{\mathbb{S}^n}r_3^2d\sigma_\xi,
\end{align*}
where the parameter $\sigma_0 > 0$ will be determined later. To summarize, we have
\begin{align*}
&B_{n,s}\int_{\mathbb{S}^n}\int_{\mathbb{S}^n}\frac{|r(\xi)-r(\eta)|^2}{|\xi-\eta|^{n+2s}}d\sigma_\xi d\sigma_\eta+ \int_{\mathbb{S}^n}|1+r|^2d\sigma_\xi-(\int_{\mathbb{S}^{n}}(1+r)^{q}\,d\sigma_\xi)^{2/q}\\
&\ \ \geq \theta \varepsilon_0\big(B_{n,s}\int_{\mathbb{S}^n}\int_{\mathbb{S}^n}\frac{|r(\xi)-r(\eta)|^2}{|\xi-\eta|^{n+2s}}d\sigma_\xi
 d\sigma_\eta+\int_{\mathbb{S}^n}r^2d\sigma_\xi\big) + \sum\limits_{k=1}^{3}I_k.
\end{align*}

In the following, we will show that $I_3$, $I_1$ and $I_2$ are nonnegative respectively. More precisely, we will prove that
for any fixed $0<\varepsilon_0<1/12$, and
\begin{align}\label{paremeters}
\varepsilon_1=\frac 1 2(\frac{1}{3}-4\varepsilon_0),~~~\varepsilon_2=\frac{1-4\varepsilon_0}{8},~~~\sigma_0=\frac{2}{q}\varepsilon_2,
\end{align}
there exists a constant $\delta_0\in(0,1)$, independent of $s$ and $n$, such that for  all $-1\leq r\in H^{s}(\mathbb{S}^n)$ satisfying
$$\left(\int_{\mathbb{S}^n}|u|^{\frac{2n}{n-2s}}d\sigma_\xi\right)^{\frac{n-2s}{n}}\leq \delta_0~~~\textit{and}~~~\int_{\mathbb{S}^n}r d\sigma_\xi=0=\int_{\mathbb{S}^n}w_jrd\sigma_\xi, ~~j=1,\cdots,n+1,$$
there holds $I_i\geq 0$, $i=1,2,3$.

\subsection{Bound on $I_3$}
Let us prove if we choose $\varepsilon_2$ and $\sigma_0$ as in (\ref{paremeters}), then
\begin{equation*}\begin{split}
& I_3=(1-\theta\varepsilon_0)\big(B_{n,s}\int_{\mathbb{S}^n}\int_{\mathbb{S}^n}\frac{|r_3(\xi)-r_3(\eta)|^2}{|\xi-\eta|^{n+2s}}d\sigma_\xi
 d\sigma_\eta+\int_{\mathbb{S}^n}r_3^2d\sigma_\xi\big)\\
 &\ \ -\frac{2}{q}(1+\varepsilon_2\theta)\int_{\mathbb{S}^n}r_3^qd\sigma_\xi-\sigma_0\theta\int_{\mathbb{S}^n}r_3^2d\sigma_\xi\geq 0.
\end{split}\end{equation*}
Recalling $\sigma_0=\frac{2}{q}\varepsilon_2$ and using the fact $\|r_3\|_{L^q(\mathbb{S}^n)}^q\leq \|r_3\|_{L^q(\mathbb{S}^n)}^2$, $ \|r_3\|_{L^2(\mathbb{S}^n)}\leq \|r_3\|_{L^q(\mathbb{S}^n)}$ and the Sobolev inequality, we can write
\begin{equation*}\begin{split}
I_3&\geq\left((1-\theta\varepsilon_0)\big(B_{n,s}\int_{\mathbb{S}^n}\int_{\mathbb{S}^n}\frac{|r_3(\xi)-r_3(\eta)|^2}{|\xi-\eta|^{n+2s}}d\sigma_\xi
 d\sigma_\eta+\int_{\mathbb{S}^n}r_3^2d\sigma_\xi\big)\right)\\
 &\ \ -\frac{2}{q}(1+2\varepsilon_2\theta)\big(\int_{\mathbb{S}^n}r_3^qd\sigma_\xi\big)^{\frac{2}{q}}\\
&\ \ \geq (1-\theta\varepsilon_0)\big(\int_{\mathbb{S}^n}r_3^qd\sigma_\xi\big)^{\frac{2}{q}}-\frac{2}{q}(1+2\varepsilon_2\theta)\big(\int_{\mathbb{S}^n}r_3^qd\sigma_\xi\big)^{\frac{2}{q}}\\
&=\left(1-\varepsilon_0 \frac{4s}{n-2s}-(1-\frac{2s}{n})(1+2\varepsilon_2\frac{4s}{n-2s})\right)\big(\int_{\mathbb{S}^n}r_3^qd\sigma_\xi\big)^{\frac{2}{q}}.
\end{split}\end{equation*}
Since $1-\varepsilon_0 \frac{4s}{n-2s}-(1-\frac{2s}{n})(1+2\varepsilon_2\frac{4s}{n-2s})\geq \frac{s(n-4s)}{n(n-2s)}>0$ when $4\varepsilon_0+8\varepsilon_2\leq 1$ and
$n>6s$. Then $I_3$ is nonnegative.

\subsection{Bound on $I_1$}
In this subsection, we will prove $I_1\geq 0$. Let $\tilde{Y}_{l,m}$ be the $L^{2}$-normalized spherical harmonics with respect to the uniform probability measure on the sphere for any $m=1,2,...,N(n,l)$, where
$$N(n,0)=1~~\text{and}~~N(n,l)=\frac{(2l+n-1)\Gamma(l+n-1)}{\Gamma(l+1)\Gamma(n)}.$$
 Let $\widetilde{r_1}$ be the orthogonal projection of $r_1$ onto the space of spherical harmonics of degree $\geq 2$, that is,
\begin{align*}
\widetilde{r_1}=r_1 - \int_{\mathbb{S}^{n}}r_1d\sigma_\xi - \sum_{j=1}^{n+1}\big(\tilde{Y}_{1,j}\int_{\mathbb{S}^{n}}r_1\tilde{Y}_{1,m}d\sigma_\xi\big).
\end{align*}
By the orthogonality of spherical harmonics, we have
\begin{align}\label{sph of r_1}
\int_{\mathbb{S}^n}r_1^2d\sigma_\xi=\int_{\mathbb{S}^n}\widetilde{r_1}^2d\sigma_\xi+(\int_{\mathbb{S}^n}r_1d\sigma_\xi)^2+\sum_{j=1}^{n+1}\big(\int_{\mathbb{S}^{n}}r_1\tilde{Y}_{1,m}d\sigma_\xi\big)^2.
\end{align}
Now we need the following Aronszajn-Smith formula by  Beckner (Lemma 10 in~\cite{Be2015}).
\begin{lemma}\label{A-S formula}
Let $F=\sum_{l,m} c_{l,m}\tilde{Y}_{l,m}$, $0<s<\min\{1,n/2\}$, then
$$B_{n,s}\iint_{\mathbb{S}^n\times \mathbb{S}^n}\frac{|F(\xi)-F(\eta)|^2}{|\xi-\eta|^{n+2s}}d\sigma_\xi d\sigma_\eta=\sum_{l=1}^{\infty}\left[A_{n,s}(l)-1\right]\sum_{m=1}^{N(n,l)}c_{l,m}^2,$$
where $B_{n,s}=2^{n+2s}\frac{s\Gamma(\frac{n}{2}+1)\Gamma(\frac{n}{2}-s)}{\Gamma(n+1)\Gamma(1-s)}$ and $A_{n,s}(l)=\frac{\Gamma(\frac{n}{2}-s)\Gamma(\frac{n}{2}+s+l)}{\Gamma(\frac{n}{2}+s)\Gamma(\frac{n}{2}-s+l)}$.
\end{lemma}
Then using (\ref{sph of r_1}), Aronszajn-Smith formula and the fact that $A_{n,s}(l)$ is increasing on $l$, we can estimate $I_1$ as follows,

\begin{align*}
I_1&=(1-\theta\varepsilon_0)B_{n,s}\int_{\mathbb{S}^n}\int_{\mathbb{S}^n}\frac{|r_1(\xi)-r_1(\eta)|^2}{|\xi-\eta|^{n+2s}}d\sigma_\xi+(2-q-\varepsilon_1\theta-\varepsilon_0\theta)\int_{\mathbb{S}^n}r_1^2d\sigma_\xi\\
 &\ \  +\sigma_0\theta\int_{\mathbb{S}^{n}}({r_2}^{2}+{r_3}^{2})d\sigma_\xi\\
&\geq (1-\theta\varepsilon_0)(A_{n,s}(2)-1)\int_{\mathbb{S}^n}\widetilde{r_1}^2d\sigma_\xi+(1-\theta\varepsilon_0)(A_{n,s}(1)-1)\sum_{j=1}^{n+1}\big(\int_{\mathbb{S}^n}(r_1\tilde{Y}_{1,j})^2d\sigma_\xi\big)\\
&\ \ -(\frac{4s}{n-2s}+(\varepsilon_1+\varepsilon_0)\frac{4s}{n-2s})\int_{\mathbb{S}^n}\widetilde{r_1}^2d\sigma_\xi-(\frac{4s}{n-2s}+(\varepsilon_1+\varepsilon_0)\frac{4s}{n-2s})(\int_{\mathbb{S}^n}r_1d\sigma_\xi)^2\\
&\ \ -(\frac{4s}{n-2s}+(\varepsilon_1+\varepsilon_0)\frac{4s}{n-2s})\sum_{j=1}^{n+1}\big(\int_{\mathbb{S}^{n}}r_1\tilde{Y}_{1,j}d\sigma_\xi\big)^2+\frac{4s}{n-2s}\sigma_0\int_{\mathbb{S}^{n}}({r_2}^{2}+{r_3}^{2})d\sigma_\xi,
\end{align*}
where $$A_{n,s}(1)=\frac{\Gamma(\frac{n}{2}-s)\Gamma(\frac{n}{2}+s+1)}{\Gamma(\frac{n}{2}+s)\Gamma(\frac{n}{2}-s+1)}=1+\frac{4s}{n-2s}$$
and $$A_{n,s}(2)=\frac{\Gamma(\frac{n}{2}-s)\Gamma(\frac{n}{2}+s+2)}{\Gamma(\frac{n}{2}+s)\Gamma(\frac{n}{2}-s+2)}=(1+\frac{4s}{n-2s})\cdot(1+\frac{4s}{n-2s+2}).$$
Since
$$(1-\theta\varepsilon_0)(A_{n,s}(2)-1)-\frac{4s}{n-2s}(1+\varepsilon_0+\varepsilon_1)\geq \frac{4s}{n-2s}(\frac{1}{3}-4\varepsilon_0-\varepsilon_1)>0$$
when $n\geq 6s$, then we have
\small{\begin{equation}\label{I_1}\begin{split}
&I_1(\frac{4s}{n-2s})^{-1}\\
&\ \ \geq (\frac{1}{3}-4\varepsilon_0-\varepsilon_1)\int_{\mathbb{S}^n}\widetilde{r_1}^2d\sigma_\xi-(1+\varepsilon_0+\varepsilon_1)
(\int_{\mathbb{S}^n}r_1d\sigma_\xi)^2-(1+\varepsilon_1+\varepsilon_0)\sum_{j=1}^{n+1}\big(\int_{\mathbb{S}^{n}}r_1\tilde{Y}_{1,j}d\sigma_\xi\big)^2\\
&\ \ \ \ \ +\sigma_0\int_{\mathbb{S}^{n}}({r_2}^{2}+{r_3}^{2})d\sigma_\xi\\
&\ \ =(\frac{1}{3}-4\varepsilon_0-\varepsilon_1)\int_{\mathbb{S}^n}r_1^2d\sigma_\xi
-(\frac{4}{3}-3\varepsilon_0)(\int_{\mathbb{S}^n}r_1d\sigma_\xi)^2-(\frac{4}{3}-3\varepsilon_0)\sum_{j=1}^{n+1}\big(\int_{\mathbb{S}^{n}}r_1\tilde{Y}_{1,j}d\sigma_\xi\big)^2\\
&\ \ \ \ \ +\sigma_0\int_{\mathbb{S}^{n}}(r_2^{2}+r_3^{2})d\sigma_\xi.\\
\end{split}\end{equation}}

Next we will show that $I_1 \geq 0$.
Let $Y$ be one of the functions 1 and $\sum_{j=1}^{n+1}a_j\tilde{Y}_{1,j}$, where $a_j \in \mathbb{R}$. Using
the fact $$\int_{\mathbb{S}^{n}}Yrd\sigma_\xi=0,$$
we can get
\begin{equation}\begin{split}\label{adint1}
\left(\int_{\mathbb{S}^{n}}Yr_1\,d\sigma_\xi\right)^{2}& = \left(\int_{\mathbb{S}^{n}}Y(r_2+r_3)\,d\sigma_\xi \right)^{2}\\
 &\leq \|Y\|^{2}_{L^{4}(\mathbb{S}^{n})}\left(\int_{\{r_2+r_3>0\}}d\sigma_\xi\right)^{\frac{1}{2}} \|r_2+r_3\|^{2}_{L^{2}(\mathbb{S}^{n})}.
\end{split}\end{equation}
Since $\{r_2+r_3>0\} \subset \{r_1 \geq \gamma\}$, we have
\begin{equation}\label{adint2}\begin{split}
\int_{\{r_2+r_3>0\}}d\sigma_\xi \leq \int_{\{r_1\geq \gamma\}}d\sigma_\xi \leq \frac{1}{\gamma^{2}}\int_{\mathbb{S}^{n}}r_1^{2}\,d\sigma_\xi = \frac{1}{\gamma^{2}}\|r_1\|^{2}_{L^{2}(\mathbb{S}^{n})}.
\end{split}\end{equation}
By \eqref{adint1} and \eqref{adint2}, we derive that
\begin{align*}
\left(\int_{\mathbb{S}^{n}}Yr_1\,d\sigma_\xi\right)^{2} \leq \|Y\|^{2}_{L^{4}(\mathbb{S}^{d})}\frac{1}{\gamma}\|r_1\|_{L^{2}(\mathbb{S}^{d})} \|r_2+r_3\|^{2}_{L^{2}(\mathbb{S}^{n})}.
\end{align*}
Since $$\|r_2+r_3\|^{2}_{L^{2}(\mathbb{S}^{n})} \leq 2\int_{\mathbb{S}^{n}}\left(r_2^{2}+r_3^{2}\right)d\sigma_\xi\leq 2 \|r\|_{L^2(\mathbb{S}^n)}^2$$
and $$ \|r\|_{L^2(\mathbb{S}^n)}^2\leq \|r\|_{L^q(\mathbb{S}^n)}^2\leq \delta_0.$$
Combining the above estimates, we get that
\begin{equation}\label{int1}
\left(\int_{\mathbb{S}^{n}}Yr_1\,d\sigma_\xi\right)^{2}\leq \|Y\|^{2}_{L^{4}(\mathbb{S}^{n})}\frac{\sqrt{2\delta_0}}{\gamma} \|r_1\|_{L^{2}(\mathbb{S}^{n})}\left(\int_{\mathbb{S}^{n}}\left(r_2^{2}+r_3^{2}\right)d\sigma_\xi\right)^{\frac{1}{2}}.
\end{equation}
If $Y = 1$, \eqref{int1} directly gives that
\begin{align}\label{adint3}
\left(\int_{\mathbb{S}^{n}}r_1\,d\sigma_\xi\right)^{2} \leq \frac{\sqrt{2\delta_0}}{\gamma}\|r_1\|_{L^{2}(\mathbb{S}^{n})}\left(\int_{\mathbb{S}^{n}}({r_2}^{2}+{r_3}^{2})d\sigma_\xi\right)^{1/2}.
\end{align}
If $Y = \sum_{j=1}^{n+1}a_j\tilde{Y}_{1,j}$, then a quick computation gives
\begin{align*}
\|Y\|^{4}_{L^{4}(\mathbb{S}^{n})}= (n+1)^2\frac{\int_0^\pi\cos^{4}\theta\sin^{n-1}\theta d\theta}{\int_0^\pi\sin^{n-1}\theta d\theta}(\sum_{j=1}^{n+1}a_j^2)^{2} = \frac{3(\sum_{j=1}^{n+1}a_j^2)^{2}(n+1)^2}{(n+3)(n+1)} \leq 3(\sum_{j=1}^{n+1}a_j^2)^{2}.
\end{align*}
Pick $a_j=\int_{\mathbb{S}^{n}}  \tilde{Y}_{1,j}r_1d\sigma_\xi$, then it follows that
\begin{equation}\label{adint4}\begin{split}
\sum_{j=1}^{n+1}\big(\int_{\mathbb{S}^{n}}r_1\tilde{Y}_{1,j}d\sigma_\xi\big)^2&= (\sum_{j=1}^{n+1}a_j^2)^{-1}(\int_{\mathbb{S}^{n}}Yr_1\,d\sigma_\xi)^{2}
 \leq\sqrt{3}\frac{\sqrt{2\delta_0}}{\gamma}\|r_1\|_{L^{2}(\mathbb{S}^{n})}(\int_{\mathbb{S}^{n}}r_2^{2}+r_3^{2}\,d\sigma_\xi)^{1/2}.
\end{split}\end{equation}
Gathering \eqref{I_1}, \eqref{adint3} and \eqref{adint4}, we conclude that
\small{\begin{equation}\begin{split} I_1
&\geq (\frac{1}{3}-4\varepsilon_0-\varepsilon_1)\int_{\mathbb{S}^n}r_1^2d\sigma_\xi+\sigma_0\int_{\mathbb{S}^{n}}(r_2^{2}+r_3^{2})d\sigma_\xi\\
&\ \ -(\frac{4}{3}-3\varepsilon_0)\frac{\sqrt{2\delta_0}}{\gamma}\|r_1\|_{L^{2}(\mathbb{S}^{n})}\left(\int_{\mathbb{S}^{n}}r_2^{2}+r_3^{2}d\sigma_\xi\right)^{1/2}\\
&\ \ -(\frac{4}{3}-3\varepsilon_0)\sqrt{3}\frac{\sqrt{2\delta_0}}{\gamma}\|r_1\|_{L^{2}(\mathbb{S}^{n})}\left(\int_{\mathbb{S}^{n}}r_2^{2}+r_3^{2}d\sigma_\xi\right)^{1/2}.
\end{split}\end{equation}}
Then by the Cauchy-Swartz inequality, $I_1$ is nonnegative if
\begin{align}\label{ine for I_1}
2\sqrt{1/3-4\varepsilon_0-\varepsilon_1}\sqrt{\sigma_0}\geq (1+\sqrt{3})(4/3-\varepsilon_0)\frac{\sqrt{2\delta_0}}{\gamma}.
\end{align}
Recalling that $\sigma_0=\frac{2}{q}\varepsilon_2$ and $\gamma=\frac {\varepsilon_1}{2}$, then we have
 $$\frac{2}{q}\varepsilon_2\geq \left((1+\sqrt{3})(4/3-\varepsilon_0)\right)^2\frac{2\delta_0}{(1/3-4\varepsilon_0-\varepsilon_1)\varepsilon_1^2}.$$
Denote
\begin{align}\label{delta1}
\delta_1:= \frac{(1/3-4\varepsilon_0-\varepsilon_1)\varepsilon_1^2\varepsilon_2}{3\left((1+\sqrt{3})(4/3-\varepsilon_0)\right)^2}.
\end{align}
Let us choose $\delta_0\leq \delta_1$, then
$$\delta_0\leq \delta_1\leq \frac{(1/3-4\varepsilon_0-\varepsilon_1)\varepsilon_1^2\varepsilon_2}{q\left((1+\sqrt{3})(4/3-\varepsilon_0)\right)^2},$$
which will make inequality (\ref{ine for I_1}) hold. And it is obvious that $\delta_1$ is independent of $s$ and $n$.

\subsection{Bound on $I_2$}
In this subsection, we will prove $I_2\geq 0$. From \cite{DEFFL}(see inequality (25) in \cite{DEFFL}), we know that for any $L^{2}$-normalized spherical harmonic $Y$ of degree $k \in \mathbb{N}$
\begin{align}
\int_{\mathbb{S}^{n}}Yr_2\,d\sigma_{\xi} \leq 3^{\frac{k}{2}}\gamma^{-\frac{q}{4}}\delta_0^{\frac{q}{8}}\|r_2\|_{L^{2}(\mathbb{S}^{n})}.
\end{align}
Let us denote by $\pi_{k}r_2$ the projection of $r_2$ onto spherical harmonics of degree $k$, and let $Y=\frac{\pi_{k}r_2}{\|\pi_{k}r_2\|_2}$, then there holds
\begin{align*}
\|\pi_kr_2\|_{L^{2}(\mathbb{S}^{n})} \leq 3^{\frac{k}{2}}\gamma^{-\frac{q}{4}}\delta_0^{\frac{q}{8}}\|r_2\|_{L^{2}(\mathbb{S}^{n})}.
\end{align*}
For any $K \in \mathbb{N}$, we denote by $\Pi_Kr_2:= \sum_{k<K}\pi_{k}r_2$ the projection of $r_2$ onto spherical harmonics of degree less than $K$, then we can derive that
\begin{equation}\begin{split}\label{int2}
\|\Pi_Kr_2\|_{L^{2}(\mathbb{S}^{n})}&= (\sum_{k<K}\|\pi_kr_2\|^{2}_{L^{2}(\mathbb{S}^{n})})^{1/2}\\
 &\leq \gamma^{-\frac{q}{4}}\delta_0^{\frac{q}{8}}\|r_2\|_{L^{2}(\mathbb{S}^{n})}\sqrt{\sum_{k<K}3^{k}}\\
  &\leq 3^{\frac{K}{2}}\gamma^{-\frac{q}{4}}\delta_0^{\frac{q}{8}}\|r_2\|_{L^{2}(\mathbb{S}^{n})}.
\end{split}\end{equation}
By Lemma \ref{A-S formula}, we can write
\begin{align*}
B_{n,s}\int_{\mathbb{S}^n}\int_{\mathbb{S}^n}\frac{|r_2(\xi)-r_2(\eta)|^2}{|\xi-\eta|^{n+2s}}d\sigma_\xi
 d\sigma_\eta=\sum_{j=0}^{+\infty}(A_{n,s}(j)-1)\|\pi_jr_2\|_{L^2(\mathbb{S}^n)}^2.
\end{align*}
This together with \eqref{int2} gives that
\begin{align*}
&B_{n,s}\int_{\mathbb{S}^n}\int_{\mathbb{S}^n}\frac{|r_2(\xi)-r_2(\eta)|^2}{|\xi-\eta|^{n+2s}}d\sigma_\xi
 d\sigma_\eta\\
 &\ \ \geq \sum_{j=K}^{+\infty}(A_{n,s}(j)-1)\|\pi_jr_2\|_{L^2(\mathbb{S}^n)}^2\\
 &\ \ \geq (A_{n,s}(K)-1)\left(\|r_2\|_{L^2(\mathbb{S}^n)}^2-\|\Pi_Kr_2\|_{L^{2}(\mathbb{S}^{n})}^2\right)\\
 &\ \ \geq (A_{n,s}(K)-1)\left(1-3^{K}\gamma^{-\frac{q}{2}}\delta_0^{\frac{q}{4}}\right)\|r_2\|^2_{L^{2}(\mathbb{S}^{n})}.
\end{align*}
Then we conclude that
\begin{align*}
I_2&=(1-\theta\varepsilon_0)\big(B_{n,s}\int_{\mathbb{S}^n}\int_{\mathbb{S}^n}\frac{|r_2(\xi)-r_2(\eta)|^2}{|\xi-\eta|^{n+2s}}d\sigma_\xi
 d\sigma_\eta+\int_{\mathbb{S}^n}r_2^2d\sigma_\xi\big)\\
&\ \ +(1-q-C_{\varepsilon_1, \varepsilon_2}\theta-\sigma_0\theta)\int_{\mathbb{S}^n}r_2^2d\sigma_\xi\\
&\geq (1-\frac{4s}{n-2s}\varepsilon_0)(A_{n,s}(K)-1)\left(1-3^{K}\gamma^{-\frac{q}{2}}\delta^{\frac{q}{4}}\right)\|r_2\|_{L^{2}(\mathbb{S}^{n})}\\
&\ \ -\frac{4s}{n-2s}\left(1+C_{\varepsilon_1, \varepsilon_2}+\sigma_0+\varepsilon_0\right)\|r_2\|_{L^{2}(\mathbb{S}^{n})}.
\end{align*}
Choose $\delta_0$ satisfying $$\left(1-3^{K}\gamma^{-\frac{q}{2}}\delta_0^{\frac{q}{4}}\right)\geq \frac{2}{3},$$
that is
\begin{align}\label{delta}
\delta_0\leq 3^{(-K-1)\frac{4}{q}}(\varepsilon_1/2)^2.
\end{align}

Denote
\begin{align}\label{delta2}
\delta_2:=3^{2(-K-1)}(\varepsilon_1/2)^2
\end{align}
and we choose $\delta_0\leq \delta_2$, then (\ref{delta}) holds since $2\leq q$. Noticing
\begin{equation}\begin{split}
A_{n,s}(K)-1&=\frac{\Gamma(\frac{n}{2}-s)\Gamma(\frac{n}{2}+s+K)}{\Gamma(\frac{n}{2}+s)\Gamma(\frac{n}{2}-s+K)}-1\\
&=\frac{(\frac{n}{2}+s)(\frac{n}{2}+s+1)\cdot\cdot\cdot(\frac{n}{2}+s+K-1)}{(\frac{n}{2}-s)(\frac{n}{2}-s+1)\cdot\cdot\cdot(\frac{n}{2}-s+K-1)}-1\\
&=\left(1+\frac{4s}{n-2s}\right)\left(1+\frac{4s}{n-2s+2}\right)\cdot\cdot\cdot\left(1+\frac{4s}{n-2s+2K-2}\right)-1\\
&\geq s\sum_{j=1}^{K}\left(\frac{4}{n-2s+2j-2}\right),
\end{split}\end{equation}
and $$\sum_{j=1}^{+\infty}\left(\frac{4}{n-2s+2j-2}\right)=+\infty,$$
we can choose $K$ sufficiently large, which is independent of $s$ such that
$$\frac{2}{3}(1-\frac{4s}{n-2s}\varepsilon_0)(A_{n,s}(k)-1)\geq \frac{4s}{n-2s}\left(1+C_{\varepsilon_1, \varepsilon_2}+\sigma_0+\varepsilon_0\right).$$
In fact since $\frac{4s}{n-2s}<1$, $n-2s>\frac{2n}{3}$ and $\sigma_0=\frac{2}{q}\varepsilon_2\leq \varepsilon_2$, we only need to choose $K$ big enough such that
$$\frac{2}{3}(1-\varepsilon_0)\sum_{j=1}^{K}\left(\frac{4}{n+2j-2}\right)\geq \frac{6}{n}\left(1+C_{\varepsilon_1, \varepsilon_2}+\varepsilon_2+\varepsilon_0\right).$$
Now it is obvious that $K$ is independent of $s$. In fact, we can choose $K$ independent of $n$ as well. That is because by the fact $\frac{n}{n+2j-2}\geq \frac{1}{1+2j-2}$ we only need to choose $K$ such that
$$\frac{2}{3}(1-\varepsilon_0)\sum_{j=1}^{K}\left(\frac{4}{1+2j-2}\right)\geq 6\left(1+C_{\varepsilon_1, \varepsilon_2}+\varepsilon_2+\varepsilon_0\right).$$
Then we obtain that $I_3\geq 0$ when $\delta_0\leq \delta_2$.

To summarize, if we choose
\begin{align}\label{delta0}
\delta_0=\min\{\delta_1,\delta_2\},
\end{align}
 where $\delta_1$
 and $\delta_2$ are defined by (\ref{delta1}) and (\ref{delta2}) which are independent of $s$ and $n$,  we can
prove $I_i\geq 0$ for $i=1,2,3$, which implies Lemma \ref{local sta}.

\section{Stability of fractional Sobolev inequalities with the optimal asymptotic lower bound}
In this section we will set up the stability of the fractional Sobolev inequalities with optimal asymptotic lower bound by adapting the strategy of the author's previous work~\cite{CLT1}. For reader's convenience, we give the outline of the proof here and for more details we refer the reader to see \cite{CLT1}. The proof is divided into five steps.
\vskip0.1cm

In Step 1, denote $2^\ast_s=\frac{2n}{n-2s}$ and define
\begin{equation*}\begin{split}
\nu(\delta)=\inf\left\{\mathcal{S}_{s}(f): f\geq 0, \inf_{g\in M_S}\|A_{2s}^{1/2}(f-g)\|_2^2\leq \delta\|A_{2s}^{1/2}f\|_2^2\right\},
\end{split}\end{equation*}
where $$\mathcal{S}_{s}(f)=\frac{\|A_{2s}^{1/2}f\|_2^2-\mathcal{S}_{s,n}\|f\|^2_{2^\ast_s}}{\inf\limits_{h\in M_{S}}\|A_{2s}^{1/2}(f-h)\|^2_{2}}.$$
We can establish the relation between the local stability for Hardy-Littlewood-Sobolev inequality and fractional sobolev inequality on the sphere for non-negative functions. This can be proved by applying the duality method of Carlen from \cite{Carlen}.
Denote
\begin{align*}
\mathcal{S}_{HLS}(g)=\frac{\|g\|^2_{\frac{2n}{n+2s}}-\mathcal{S}_{s,n}\|\mathcal{P}_{2s}g\|_2^2}{\inf\limits_{h\in M_{HLS}}\|g-h\|^2_{\frac{2n}{n+2s}}},
\end{align*}
where the operator $\mathcal{P}_{2s}$ is given by
$$\mathcal{P}_{2s}g(\eta):=\frac{\Gamma(\frac{n-2s}{2})}{2^{2s}\pi^{\frac{n}{2}}\Gamma(s)}\int_{\mathbb{S}^n}\frac{g(\omega)}{|\eta-\omega|^{n-2s}}d\omega,$$
and
$$M_{HLS}=\{c(1-\xi\cdot\omega)^{-\frac{n+2s}{2}}:\xi\in B^{n+1},c\in\mathbb{R}\},$$
is the optimizer set of sharp Hardy-Littlewood-Sobolev inequality on the sphere.
Define
$$\mu(\delta)=\inf\{\mathcal{S}_{HLS}(g): g>0, \inf_{h\in M_{HLS}}\|g-h\|^2_{\frac{2n}{n+2s}}\leq \delta\|g\|^2_{\frac{2n}{n+2s}}\},$$
we can then prove that $$\mu(\frac{\delta}{2})\geq \frac{1}{2}\frac{n-2s}{n+2s}\min\{\nu(\delta)\frac{n-2s}{n+2s},1\}.$$
\vskip0.3cm
In Step 2, we establish the relationship between the local stability and the global stability of Hardy-Littlewood-Sobolev inequality for positive functions.
The proof is in the spirit of the work  by Figalli et al. in \cite{DEFFL} establishing the relationship between the local stability and the global stability of fractional Sobolev inequality for positive functions. We apply the rearrangement flow technique for HLS integral replacing the rearrangement flow technique for $L^2$ integral of the gradient. More precisely, we can prove that for any $0\leq g \in L^{\frac{2n}{n+2s}}(\mathbb{R}^n)$ satisfying
$$\inf\limits_{h\in M_{HLS}}\|g-h\|^2_{\frac{2n}{n+2s}}> \delta\|g\|^2_{\frac{2n}{n+2s}},$$ there holds
$$\mathcal{S}_{HLS}(g)\geq \delta \mu(\delta).$$

\vskip 0.3cm

In Step 3, by Lemma~\ref{local stability}, we have
$$\nu(\frac{\delta_0}{1+\delta_0})\geq \frac{4s}{n-2s}\varepsilon_0.$$
Combining the Step 1 and Step 2, we derive the stability of HLS inequality with the optimal asymptotic lower bound for positive function
$$\mathcal{S}_{HLS}(g)\geq \frac{\delta_0}{1+\delta_0}\frac{1}{4}\frac{n-2s}{n+2s}\min\{4s\frac{\varepsilon_0}{n+2s},1\}:=K_{n,s}.$$
\vskip0.3cm

In Step 4, we establish the stability of HLS inequality with the optimal lower bound for general function. This is achieved by establishing the relation between the stability of HLS inequality for positive function and general function. Denoting by $C_{HLS}$ and $C^{pos}_{HLS}$ the optimal constants of stability of HLS inequality for positive function and general function respectively, we prove that
$$C_{HLS}\geq \frac{1}{2}\min\{C^{pos}_{HLS},\min\{2^{\frac{n+2s}{n}}-2,1\}\}.$$
\vskip0.3cm

In Step 5, by dual method we can prove that
$$\inf\limits_{f\in H^s(\mathbb{S}^n)\setminus M_s}\mathcal{S}_{S}(f)\geq \frac{\min\{K_{n,s},\min\{2^{\frac{n+2s}{n}}-2,1\}\}}{4}.$$
Let us denote
\begin{align}\label{beta}
\beta_{s,n}=n\frac{\min\{K_{n,s},\min\{2^{\frac{n+2s}{n}}-2,1\}\}}{4s}.
\end{align}
Then it follows that $\lim\limits_{(s,n)\rightarrow (0, +\infty)}\beta_{s,n}=\frac{\delta_0\varepsilon_0}{4(1+\delta_0)}$. Then we  accomplish the proof of Theorem \ref{sta of so}.

\section{Global stability of log-Sobolev stability on the sphere }
In the section, we will establish the global stability of log-Sobolev inequality on the sphere by the stability of fractional Sobolev inequality with the optimal asymptotic lower bound and approximation method. Beckner~\cite{Be1997} found that the log-Sobolev inequality on the sphere can be obtained as a limit of sharp Sobolev inequalities. Since we have obtained the optimal asymptotic lower bound for stability of the fractional Sobolev inequality in Theorem \ref{sta of so}, this makes it possible to derive the global stability of log-Sobolev inequality via the end-point differentiation method. In fact, we will  do the limiting process to the stability of fractional Sobolev inequality by letting $s\rightarrow 0$.
\vskip0.1cm

Recalling the stability of fractional Sobolev inequalities with the optimal asymptotic lower bound (Theorem~\ref{sta of so}):
$$\|A^{1/2}_{2s}u\|_2^2-\frac{\Gamma(n/2+s)}{\Gamma(n/2-s)}|\mathbb{S}^n|^{\frac{2s}{n}}\|u\|_{2^\ast_s}^2\geq \frac{s}{n}\beta_{s,n}\inf_{h\in M_{s}}\|A^{1/2}_{2s}(u-h)\|_{2}^2.$$
Then by the fractional Sobolev inequality, there holds
\begin{align}\label{sta of so1}
\|A^{1/2}_{2s}u\|_2^2-\frac{\Gamma(n/2+s)}{\Gamma(n/2-s)}|\mathbb{S}^n|^{\frac{2s}{n}}\|u\|_{2^\ast_s}^2\geq \frac{s}{n}\mathcal{S}_{s,n}\beta_{s,n} \inf_{h\in M_{s}}\|u-h\|_{2^\ast_s}^2.
\end{align}
In order to prove Theorem \ref{sta of lg}, we first prove that when $s\rightarrow 0$, the $L^{2^\ast_s}$ distance $\inf_{h\in M_s}\|f-h\|_{2^\ast_s}$ can be bounded from below by the $L^{2}$ distance $\inf_{h\in M_0}\|f-h\|_{2}$.
\begin{lemma}\label{s goes 0}
If $f\in C(\mathbb{S}^n)$, then
$$\lim_{s\rightarrow 0}\inf_{h\in M_s}\|u-h\|_{2^\ast_s}^2 \geq \inf_{h\in M_0}\|u-h\|_{2}^2.$$
\end{lemma}
\begin{proof}
First we claim that for any fixed $0<s<n/2$, there exist $c_s\in \mathbb{R}$ and $\xi_s\in \mathbb{R}^{n+1}$ with $|\xi_s|< 1$ such that
$$\|u-c_s(\frac{\sqrt{1-|\xi_s|^2}}{1-\xi_s\cdot\omega})^{\frac{n-2s}{2}}\|_{2^\ast_s}=\inf_{c\in \mathbb{R},\xi\in \mathbb{R}^{n+1}, |\xi|<1}\|u-c(\frac{\sqrt{1-|\xi|^2}}{1-\xi\cdot\omega})^{\frac{n-2s}{2}}\|_{2^\ast_s}
=\inf_{h\in M_s}\|u-h\|_{2^\ast_s}.$$
Let $F(c,\xi)=\|u(\omega)-c(\frac{\sqrt{1-|\xi|^2}}{1-\xi\cdot\omega})^{\frac{n-2s}{2}}\|_{2^\ast_s}$, then $F(c,\xi)$ is a continuous function on $\mathbb{R}\times B^n$, where
$B^{n+1}=\{x\in \mathbb{R}^{n+1}||x|<1\}$. Assume $c_k\in \mathbb{R}$ and $\xi_{k}\in B^{n+1}$ such that
$$\lim_{k\rightarrow \infty}\|u-c_k(\frac{\sqrt{1-|\xi_k|^2}}{1-\xi_k\cdot\omega})^{\frac{n-2s}{2}}\|_{2^\ast_s}=\inf_{c\in \mathbb{R},\xi\in \mathbb{R}^{n+1}, |\xi|<1}\|u-c(\frac{\sqrt{1-|\xi|^2}}{1-\xi\cdot\omega})^{\frac{n-2s}{2}}\|_{2^\ast_s}.$$
Since
$$\|u\|_{2^\ast_s}\geq \inf_{c\in \mathbb{R},\xi\in \mathbb{R}^{n+1}, |\xi|<1}\|u-c(\frac{\sqrt{1-|\xi|^2}}{1-\xi\cdot\omega})^{\frac{n-2s}{2}}\|_{2^\ast_s}$$
and
$$\|u-c_k(\frac{\sqrt{1-|\xi_k|^2}}{1-\xi_k\cdot\omega})^{\frac{n-2s}{2}}\|_{2^\ast_s}\geq |c_k|-\|u\|_{2^\ast_s}$$
by the triangle inequality and the fact $\|(\frac{\sqrt{1-|\xi|^2}}{1-\xi\cdot\omega})^{\frac{n-2s}{2}}\|_{2^\ast_s}=1$, we know $\{c_k\}$ must be bounded. Then there exist $c_0\in \mathbb{R}$
and $\xi_0\in \mathbb{R}^{n+1}$ with $|\xi_0|\leq 1$ satisfying $(c_k,\xi_k)\rightarrow(c_0,\xi_0)$ (up to a subsequence). Now if $|\xi_0|<1$, by the continuity of $F(c,\xi)$ we obtain
$$F(c_0,\xi_0)=\inf_{c\in \mathbb{R},\xi\in \mathbb{R}^{n+1}, |\xi|<1}\|u-c(\frac{\sqrt{1-|\xi|^2}}{1-\xi\cdot\omega})^{\frac{n-2s}{2}}\|_{2^\ast_s}.$$
If $|\xi_0|=1$, by the Fatou lemma we have
\begin{align*}
& \|u\|_{2^\ast_s}=\|\liminf_{k\rightarrow \infty}[u-c_k(\frac{\sqrt{1-|\xi_k|^2}}{1-\xi_k\cdot\omega})^{\frac{n-2s}{2}}]\|_{2^\ast_s}\\
& \leq \liminf_{k\rightarrow \infty}\|u-c_k(\frac{\sqrt{1-|\xi_k|^2}}{1-\xi_k\cdot\omega})^{\frac{n-2s}{2}}\|_{2^\ast_s}
 =\inf_{h\in M_s}\|u-h\|_{2^\ast_s} \leq \|u\|_{2^\ast_s},
\end{align*}
which means we can choose $c_s=0$ and any $|\xi_s|<1$ such that
$$\|u-c_s(\frac{\sqrt{1-|\xi_s|^2}}{1-\xi_s\cdot\omega})^{\frac{n-2s}{2}}\|_{2^\ast_s}=\inf_{c\in \mathbb{R},\xi\in \mathbb{R}^{n+1}, |\xi|<1}\|u-c(\frac{\sqrt{1-|\xi|^2}}{1-\xi\cdot\omega})^{\frac{n-2s}{2}}\|_{2^\ast_s}.$$
Then we complete the proof of the claim.
\vskip0.1cm

Again using the fact
$$\|u\|_{2^\ast_s}\geq\|u-c_s(\frac{\sqrt{1-|\xi_s|^2}}{1-\xi_s\cdot\omega})^{\frac{n-2s}{2}}\|_{2^\ast_s}\geq |c_s|-\|u\|_{2^\ast_s}$$
we get
$$|c_s|\leq 2\|u\|_{2^\ast_s}\leq 2\|u\|_{\infty}\max\{|\mathbb{S}^n|^{1/2},1\}.$$
Since $c_s$ is bounded, we may assume $c_s\rightarrow\tilde{c}$ as $s\rightarrow 0$.
\vskip0.1cm

For the case $\tilde{c}=0$, there holds
\begin{align*}
& \lim_{s\rightarrow 0}\inf_{h\in M_s}\|u-h\|_{2^\ast_s}=\lim_{s\rightarrow 0}\|u-c_s(\frac{\sqrt{1-|\xi_s|^2}}{1-\xi_s\cdot\omega})^{\frac{n-2s}{2}}\|_{2^\ast}\\
& \geq \lim_{s\rightarrow 0}(\|u\|_{2^\ast}-|c_s|)=\|u\|_{2}\geq \inf_{h\in M_0}\|u-h\|_{2},
\end{align*}
which is the desired result.
\vskip0.1cm

If $\tilde c\neq 0$, since $|\xi_s|<1$, we may assume $\xi_s\rightarrow \tilde \xi$ with $|\tilde \xi|\leq 1$. If $|\tilde \xi|<1$, by the LDCT we
have
$$\lim_{s\rightarrow 0}\inf_{h\in M_s}\|u-h\|_{2^\ast}=\lim_{s\rightarrow 0}\|u-c_s(\frac{\sqrt{1-|\xi_s|^2}}{1-\xi_s\cdot\omega})^{\frac{n-2s}{2}}\|_{2^\ast}=\|u-\tilde c(\frac{\sqrt{1-|\tilde \xi|^2}}{1-\tilde \xi\cdot\omega})^{\frac{n}{2}}\|_2\geq \inf_{h\in M_0}\|u-h\|_{2}.$$
Then we only need to handle the case $|\tilde \xi|=1$. When $|\tilde \xi|=1$, by H\"{o}lder's inequality and the triangle inequality we have
\begin{align}\label{est1}\nonumber
& \|u-c_s(\frac{\sqrt{1-|\xi_s|^2}}{1-\xi_s\cdot\omega})^{\frac{n-2s}{2}}\|_{2^\ast}\geq |\mathbb{S}^n|^{\frac s n}\|u-c_s(\frac{\sqrt{1-|\xi_s|^2}}{1-\xi_s\cdot\omega})^{\frac{n-2s}{2}}\|_{2}\\
& \geq |\mathbb{S}^n|^{\frac s n}(\|u\|_2-|c_s|\|(\frac{\sqrt{1-|\xi_s|^2}}{1-\xi_s\cdot\omega})^{\frac{n-2s}{2}}\|_{2}).
\end{align}
Now let us prove that $\int_{\mathbb{S}^n}(\frac{\sqrt{1-|\xi_s|^2}}{1-\xi_s\cdot\omega})^{n-2s}dw\rightarrow 0$ as $s\rightarrow 0$. Choose $s$ small enough such that $n-2s>n/2$. Since $|\xi_s|\rightarrow 1$ as $s\rightarrow 0$, then by the stereographic projection we have
\begin{align}\label{est2}\nonumber
& \int_{\mathbb{S}^n}(\frac{\sqrt{1-|\xi_s|^2}}{1-\xi_s\cdot\omega})^{n-2s}dw=(\sqrt{1-|\xi_s|^2})^{n-2s}\int_{\mathbb{S}^n}(1-|\xi_s|\omega_{n+1})^{-(n-2s)}dw\\\nonumber
& =(\sqrt{1-|\xi_s|^2})^{n-2s}\int_{\mathbb{R}^n}(1-|\xi_s|\frac{1-|x|^2}{1+|x|^2})^{-(n-2s)}(\frac{2}{1+|x|^2})^ndx\\\nonumber
& =\left(2\sqrt{\frac{1+|\xi_s|}{1-|\xi_s|}}\right)^{n-2s}\int_{\mathbb{R}^n}(1+\frac{1+|\xi_s|}{1-|\xi_s|}|x|^2)^{-(n-2s)}(\frac{2}{1+|x|^2})^{2s}dx\\\nonumber
& \leq 2^n\left(\sqrt{\frac{1+|\xi_s|}{1-|\xi_s|}}\right)^{n-2s}\int_{\mathbb{R}^n}(1+\frac{1+|\xi_s|}{1-|\xi_s|}|x|^2)^{-(n-2s)}dx\\
& =2^{n}\left(\frac{1-|\xi_s|}{1+|\xi_s|}\right)^{\frac{n}{2}+s}\int_{\mathbb{R}^n}(1+|x|^2)^{-(n-2s)}dx\rightarrow 0,
\end{align}
as $s\rightarrow 0$. Then by (\ref{est1}) and (\ref{est2}) we know if $|\tilde{\xi}|=1$, there still holds
$$\lim_{s\rightarrow 0}\inf_{h\in M_s}\|u-h\|_{2^\ast}=\lim_{s\rightarrow 0}\|u-c_s(\frac{\sqrt{1-|\xi_s|^2}}{1-\xi_s\cdot\omega})^{\frac{n-2s}{2}}\|_{2^\ast}\geq \|u\|_2\geq \inf_{h\in M_0}\|u-h\|_{2},$$
which completes the proof of Lemma \ref{s goes 0}.
\end{proof}
Denote by $\mathcal{F}$ the space of functions $u$ on $\mathbb{S}^n$ with $u=\sum_{l,m}u_{l,m}Y_{l,m}$ such that only finitely many coefficients $u_{l,m}$ are nonzero. We see that $\mathcal{F}\subset \mathcal{D_{L}}$ and is dense in $\mathcal{D_{L}}$ with respect to the norm $\sqrt{\mathcal{E}(u,u)+\|u\|_2^2}$, where
 $$\mathcal{E}(u,u)=\iint_{\mathbb{S}^n\times \mathbb{S}^n}\frac{|u(\omega)-u(\eta)|^2}{|\omega-\eta|^n}d\omega d\eta.$$ We need the following Lemma.

\begin{lemma}\label{relation of lg so}
For any $u\in \mathcal{F}$, there holds
$$\iint\limits_{\mathbb{S}^n\times \mathbb{S}^n}\frac{|u(\omega)-u(\eta)|^2}{|\omega-\eta|^{n}}d\omega d\eta=\frac{4\pi^{\frac n 2}}{\Gamma(\frac{n}{2})}\lim_{s\rightarrow 0}\frac{1}{2s}\left(\|A_{2s}^{1/2}u\|_2^2-\frac{\Gamma(\frac{n}{2}+s)}{\Gamma(\frac{n}{2}-s)}\|u\|^2_2\right).$$
\end{lemma}
\begin{proof}
Let us write
\begin{align*}
& \iint\limits_{\mathbb{S}^n\times \mathbb{S}^n}\frac{|u(\omega)-u(\eta)|^2}{|\omega-\eta|^{n-2s}}d\omega d\eta
=2\iint\limits_{\mathbb{S}^n\times \mathbb{S}^n}\frac{u(\omega)(u(\omega)-u(\eta))}{|\omega-\eta|^{n-2s}}d\omega d\eta\\
& =2\iint\limits_{\mathbb{S}^n\times \mathbb{S}^n}\frac{u^2(\omega)}{|\omega-\eta|^{n-2s}}d\omega d\eta
-2\iint\limits_{\mathbb{S}^n\times \mathbb{S}^n}\frac{u(\omega)u(\eta)}{|\omega-\eta|^{n-2s}}d\omega d\eta:=I+II.
\end{align*}
Assume that $u=\sum_{l,m}u_{l,m}Y_l$ and by the Funk-Hecke formula (see \cite[Eq.~(17)]{Be1993} and also  \cite[Corollary 4.3]{FrLi2012}), we have
$$II=\frac{2^{2s}\pi^{\frac n 2}\Gamma(s)}{\Gamma(\frac{n-2s}{2})}\sum_{l,m}\frac{\Gamma(l+\frac n 2-s)}{\Gamma(l+\frac n 2+s)}u^2_{l,m}.$$
And using the fact $\int_{\mathbb{S}^n}|\omega-\eta|^{2s-n}d\eta=\frac{2^{2s}\pi^{\frac n 2}\Gamma(s)}{\Gamma(\frac n 2-s)}$, we have
$$I=\frac{2^{2s}\pi^{\frac n 2}\Gamma(s)}{\Gamma(\frac n 2-s)} \sum_{l,m}u^2_{l,m}.$$
Then
\begin{align*}
\iint\limits_{\mathbb{S}^n\times \mathbb{S}^n}\frac{|u(\omega)-u(\eta)|^2}{|\omega-\eta|^{n}}d\omega d\eta&=\lim_{s\rightarrow 0}\iint\limits_{\mathbb{S}^n\times \mathbb{S}^n}\frac{|u(\omega)-u(\eta)|^2}{|\omega-\eta|^{n-2s}}d\omega d\eta\\
& =\lim_{s\rightarrow 0} \frac{2^{2s+2}\pi^{\frac n 2}s\Gamma(s)}{\Gamma(\frac{n-2s}{2})}\sum_{l,m}\left(\frac{1}{2s}(\frac{\Gamma(\frac{n}{2}-s)}{\Gamma(\frac{n}{2}+s)}-\frac{\Gamma(l+\frac{n}{2}-s)}{\Gamma(l+\frac{n}{2}+s)})\right)u^2_{l,m}\\
& =\frac{4\pi^{\frac n 2}}{\Gamma(\frac{n}{2})}\sum_{l,m}\left(\frac{\Gamma^\prime(l+\frac{n}{2})}{\Gamma(l+\frac{n}{2})}-\frac{\Gamma^\prime(\frac{n}{2})}{\Gamma(\frac{n}{2})}\right)u^2_{l,m}.
\end{align*}

Ont the other hand,
\begin{align*}
\lim_{s\rightarrow 0}\frac{1}{2s}\left((A_{2s}u,u)-\frac{\Gamma(\frac{n}{2}+s)}{\Gamma(\frac{n}{2}-s)}\|u\|^2_2\right)&=\lim_{s\rightarrow 0}\sum_{l,m}\frac{1}{2s}\left(\frac{\Gamma(l+\frac{n}{2}+s)}{\Gamma(l+\frac{n}{2}-s)}-\frac{\Gamma(\frac{n}{2}+s)}{\Gamma(\frac{n}{2}-s)}\right)u^2_{l,m}\\
& =\sum_{l,m}\left(\frac{\Gamma^\prime(l+\frac{n}{2})}{\Gamma(l+\frac{n}{2})}-\frac{\Gamma^\prime(\frac{n}{2})}{\Gamma(\frac{n}{2})}\right)u^2_{l,m}.
\end{align*}
That is
$$\iint\limits_{\mathbb{S}^n\times \mathbb{S}^n}\frac{|u(\omega)-u(\eta)|^2}{|\omega-\eta|^{n}}d\omega d\eta=\frac{4\pi^{\frac n 2}}{\Gamma(\frac{n}{2})}\lim_{s\rightarrow 0}\frac{1}{2s}\left(\|A_{2s}^{1/2}u\|_2^2-\frac{\Gamma(\frac{n}{2}+s)}{\Gamma(\frac{n}{2}-s)}\|u\|^2_2\right).$$
\end{proof}

Now we are in position to establish the stability of log-Sobolev on the sphere. By the stability of the fractional Sobolev inequality (\ref{sta of so1}) with the optimal asymptotic lower bound and Lemma \ref{relation of lg so}, we can obtain
\begin{align*}
& \iint\limits_{\mathbb{S}^n\times \mathbb{S}^n}\frac{|u(\omega)-u(\eta)|^2}{|\omega-\eta|^{n}}d\omega d\eta\geq
\frac{4\pi^{\frac n 2}}{\Gamma(\frac{n}{2})}\lim_{s\rightarrow 0}\frac{1}{2s}\frac{\Gamma(n/2+s)}{\Gamma(n/2-s)}\left(|\mathbb{S}^n|^{\frac{2s}{n}}\|u\|_{2^\ast}^2-\|u\|^2_2\right)\\
& +\frac{4\pi^{\frac n 2}}{\Gamma(\frac{n}{2})}\lim_{s\rightarrow 0}\frac{\Gamma(n/2+s)}{\Gamma(n/2-s)}\frac{\mathcal{S}_{s,n}s\beta_s}{2s}\inf_{h\in M_{s}}\|u-h\|_{2^\ast}^2:=I+II.
\end{align*}
Since
$$\lim_{s\rightarrow 0}\frac{\|u\|_{2^\ast}^2-\|u\|_2^2}{2s}=\frac{2}{n}\int_{\mathbb{S}^n}|u|^2\ln|f|d\omega-\|u\|_2^2\frac{\ln\|u\|_2^2}{n},$$
and
$$\lim_{s\rightarrow 0}\frac{|\mathbb{S}^n|\|u\|_2^2-\|u\|_2^2}{2s}=\frac{\ln|\mathbb{S}^n|}{n}\|u\|_2^2,$$
Then
\begin{align}\label{est of I}\nonumber
& I\geq \frac{4\pi^{\frac n 2}}{\Gamma(\frac{n}{2})}\lim_{s\rightarrow 0}\frac{1}{2s}\left(|\mathbb{S}^n|^{\frac{2s}{n}}\|u\|_{2^\ast}^2-|\mathbb{S}^n|^{\frac{2s}{n}}\|u\|_2^2+|\mathbb{S}^n|\|u\|_2^2-\|u\|_2^2\right)\\
& \geq \frac{4\pi^{\frac n 2}}{\Gamma(\frac{n}{2})n}\int_{\mathbb{S}^n}|u|^2\ln\frac{|u|^2|\mathbb{S}^n|}{\|u\|_2^2}dw.
\end{align}
On the other hand, by Lemma \ref{s goes 0} and the fact $\lim\limits_{s\rightarrow 0}\mathcal{S}_{s,n}\beta_{s,n}=\frac{\delta_0\varepsilon_0}{4n(1+\delta_0)}$, we know
\begin{align}\label{est of II}
II\geq \frac{\pi^{\frac n 2}}{\Gamma(\frac{n}{2})}\frac{\delta_0\varepsilon_0}{2n(1+\delta_0)}\lim_{s\rightarrow 0}\inf_{h\in M_s}\|u-h\|_{2^\ast}^2 \geq \frac{\pi^{\frac n 2}}{\Gamma(\frac{n}{2})}\frac{\delta_0\varepsilon_0}{2n(1+\delta_0)}\inf_{h\in M_0}\|u-h\|_{2}^2.
\end{align}
Denote $\alpha_n=\frac{\pi^{\frac n 2}}{2n\Gamma(\frac{n}{2})}\frac{\delta_0\varepsilon_0}{1+\delta_0}$. Thus combining (\ref{est of I}) and (\ref{est of II}), we obtain the global stability of log-Sobolev inequality on the sphere
$$\iint\limits_{\mathbb{S}^n\times \mathbb{S}^n}\frac{|u(\omega)-u(\eta)|^2}{|\omega-\eta|^{n}}d\omega d\eta- \frac{4\pi^{\frac n 2}}{\Gamma(\frac{n}{2})n}\int_{\mathbb{S}^n}|u|^2\ln\frac{|u|^2|\mathbb{S}^n|}{\|u\|_2^2}dw
\geq \alpha_n \inf_{h\in M_0}\|u-h\|_{2}^2.$$

\section{Proof of Lemma \ref{lem split}}
In this section, we give the proof of Lemma \ref{lem split}. Setting $\widetilde{r}=r_1+r_2$. We first prove that
\begin{equation}\begin{split}
&\int_{\mathbb{S}^n}\int_{\mathbb{S}^n}\frac{|r(\xi)-r(\eta)|^2}{|\xi-\eta|^{n+2s}}d\sigma_\xi d\sigma_\eta\\
&\ \ \geq \int_{\mathbb{S}^n}\int_{\mathbb{S}^n}\frac{|\widetilde{r}(\xi)-\widetilde{r}(\eta)|^2}{|\xi-\eta|^{n+2s}}d\sigma_\xi d\sigma_\eta+\int_{\mathbb{S}^n}\int_{\mathbb{S}^n}\frac{|r_3(\xi)-r_3(\eta)|^2}{|\xi-\eta|^{n+2s}}d\sigma_\xi d\sigma_\eta.
\end{split}\end{equation}
According to the definition of $r_1$, $r_2$ and $r_3$, we know that
\begin{equation}\widetilde{r}(\xi)=\begin{cases}
&r(\xi),\ \ r(\xi)<M\\
&M,\ \ r(\xi)\geq M
\end{cases}\end{equation}
and
\begin{equation}r_3(\xi)=\begin{cases}
&0,\ \ r(\xi)<M\\
&r-M,\ \ r(\xi)\geq M
\end{cases}\end{equation}
Careful computation gives that
\begin{equation}\label{eq1}\begin{split}
&\int_{\mathbb{S}^n}\int_{\mathbb{S}^n}\frac{|r(\xi)-r(\eta)|^2}{|\xi-\eta|^{n+2s}}d\sigma_\xi d\sigma_\eta\\
&\ \ =\int_{\mathbb{S}^n}\int_{\mathbb{S}^n}\frac{|(\widetilde{r}(\xi)-\widetilde{r}(\eta))+(r_3(\xi)-r_3(\eta))|^2}{|\xi-\eta|^{n+2s}}d\sigma_\xi d\sigma_\eta\\
&\ \ =\int_{\mathbb{S}^n}\int_{\mathbb{S}^n}\frac{|\widetilde{r}(\xi)-\widetilde{r}(\eta)|^2}{|\xi-\eta|^{n+2s}}d\sigma_\xi d\sigma_\eta+\int_{\mathbb{S}^n}\int_{\mathbb{S}^n}\frac{|r_3(\xi)-r_3(\eta)|^2}{|\xi-\eta|^{n+2s}}d\sigma_\xi d\sigma_\eta\\
&\ \ \ \ +2\int_{\mathbb{S}^n}\int_{\mathbb{S}^n}\frac{(\widetilde{r}(\xi)-\widetilde{r}(\eta))\cdot(r_3(\xi)-r_3(\eta))}{|\xi-\eta|^{n+2s}}d\sigma_\xi d\sigma_\eta.
\end{split}\end{equation}
We will prove that
$$\int_{\mathbb{S}^n}\int_{\mathbb{S}^n}\frac{(\widetilde{r}(\xi)-\widetilde{r}(\eta))\cdot(r_3(\xi)-r_3(\eta))}{|\xi-\eta|^{n+2s}}d\sigma_\xi d\sigma_\eta\geq 0.$$
Denote by $g(\xi,\eta)=\left(\widetilde{r}(\xi)-\widetilde{r}(\eta)\right)\cdot\left(r_3(\xi)-r_3(\eta)\right)$ and careful computations give that
\begin{equation}\label{eq2}\begin{split}
&\int_{\mathbb{S}^n}\int_{\mathbb{S}^n}\frac{(\widetilde{r}(\xi)-\widetilde{r}(\eta))\cdot(r_3(\xi)-r_3(\eta))}{|\xi-\eta|^{n+2s}}d\sigma_\xi d\sigma_\eta\\
&\ \ =\int_{r(\xi)<M}\int_{r(\eta)<M}\frac{g(\xi,\eta)}{|\xi-\eta|^{n+2s}}d\sigma_\xi d\sigma_\eta+\int_{r(\xi)\geq M}\int_{r(\eta)\geq M}\frac{g(\xi,\eta)}{|\xi-\eta|^{n+2s}}d\sigma_\xi d\sigma_\eta\\
&\ \ \ \  +\int_{r(\xi)<M}\int_{r(\eta)\geq M}\frac{g(\xi,\eta)}{|\xi-\eta|^{n+2s}}d\sigma_\xi d\sigma_\eta+\int_{r(\xi)\geq M}\int_{r(\eta)< M}\frac{g(\xi,\eta)}{|\xi-\eta|^{n+2s}}d\sigma_\xi d\sigma_\eta\\
&\ \ =\int_{r(\xi)<M}\int_{r(\eta)\geq M}\frac{g(\xi,\eta)}{|\xi-\eta|^{n+2s}}d\sigma_\xi d\sigma_\eta+\int_{r(\xi)\geq M}\int_{r(\eta)< M}\frac{g(\xi,\eta)}{|\xi-\eta|^{n+2s}}d\sigma_\xi d\sigma_\eta\\
&\ \ =\int_{r(\xi)<M}\int_{r(\eta)\geq M}\frac{(r(\xi)-M)(0-r_3(\eta))}{|\xi-\eta|^{n+2s}}d\sigma_\xi d\sigma_\eta\\
&\ \ \ \ +\int_{r(\xi)\geq M}\int_{r(\eta)< M}\frac{(M-r(\eta))(r(\xi)-M)}{|\xi-\eta|^{n+2s}}d\sigma_\xi d\sigma_\eta\geq 0.
\end{split}\end{equation}
Combining \eqref{eq1} and \eqref{eq2}, we derive that \begin{equation}\begin{split}
&\int_{\mathbb{S}^n}\int_{\mathbb{S}^n}\frac{|r(\xi)-r(\eta)|^2}{|\xi-\eta|^{n+2s}}d\sigma_\xi d\sigma_\eta\\
&\ \ \geq \int_{\mathbb{S}^n}\int_{\mathbb{S}^n}\frac{|\widetilde{r}(\xi)-\widetilde{r}(\eta)|^2}{|\xi-\eta|^{n+2s}}d\sigma_\xi d\sigma_\eta+\int_{\mathbb{S}^n}\int_{\mathbb{S}^n}\frac{|r_3(\xi)-r_3(\eta)|^2}{|\xi-\eta|^{n+2s}}d\sigma_\xi d\sigma_\eta.
\end{split}\end{equation}
Now, in order to prove Lemma \ref{lem split}, we only need to show that
\begin{equation}\begin{split}
&\int_{\mathbb{S}^n}\int_{\mathbb{S}^n}\frac{|\widetilde{r}(\xi)-\widetilde{r}(\eta)|^2}{|\xi-\eta|^{n+2s}}d\sigma_\xi d\sigma_\eta\\
&\ \ \geq \int_{\mathbb{S}^n}\int_{\mathbb{S}^n}\frac{|r_1(\xi)-r_1(\eta)|^2}{|\xi-\eta|^{n+2s}}d\sigma_\xi d\sigma_\eta+\int_{\mathbb{S}^n}\int_{\mathbb{S}^n}\frac{|r_2(\xi)-r_2(\eta)|^2}{|\xi-\eta|^{n+2s}}d\sigma_\xi d\sigma_\eta.
\end{split}\end{equation}
It is equivalent to prove that
$$\int_{\mathbb{S}^n}\int_{\mathbb{S}^n}\frac{(r_1(\xi)-r_1(\eta))\cdot(r_2(\xi)-r_2(\eta))}{|\xi-\eta|^{n+2s}}d\sigma_\xi d\sigma_\eta\geq 0.$$
Denote by $h(\xi,\eta)=(r_1(\xi)-r_1(\eta))\cdot(r_2(\xi)-r_2(\eta))$ and careful computations give that
\begin{equation}\begin{split}
&\int_{\mathbb{S}^n}\int_{\mathbb{S}^n}\frac{(r_1(\xi)-r_1(\eta))\cdot(r_2(\xi)-r_2(\eta))}{|\xi-\eta|^{n+2s}}d\sigma_\xi d\sigma_\eta\\
&\ \ =\int_{r(\xi)<\gamma}\int_{r(\eta)<\gamma}\frac{h(\xi,\eta)}{|\xi-\eta|^{n+2s}}d\sigma_\xi d\sigma_\eta+\int_{r(\xi)\geq \gamma}\int_{r(\eta)\geq \gamma}\frac{h(\xi,\eta)}{|\xi-\eta|^{n+2s}}d\sigma_\xi d\sigma_\eta\\
&\ \ \ \  +\int_{r(\xi)<\gamma}\int_{r(\eta)\geq \gamma}\frac{h(\xi,\eta)}{|\xi-\eta|^{n+2s}}d\sigma_\xi d\sigma_\eta+\int_{r(\xi)\geq \gamma}\int_{r(\eta)< \gamma}\frac{h(\xi,\eta)}{|\xi-\eta|^{n+2s}}d\sigma_\xi d\sigma_\eta\\
&\ \ =\int_{r(\xi)<\gamma}\int_{r(\eta)\geq \gamma}\frac{h(\xi,\eta)}{|\xi-\eta|^{n+2s}}d\sigma_\xi d\sigma_\eta+\int_{r(\xi)\geq \gamma}\int_{r(\eta)< \gamma}\frac{h(\xi,\eta)}{|\xi-\eta|^{n+2s}}d\sigma_\xi d\sigma_\eta\\
&\ \ =\int_{r(\xi)<\gamma}\int_{r(\eta)\geq \gamma}\frac{(r(\xi)-\gamma)(0-r_2(\eta))}{|\xi-\eta|^{n+2s}}d\sigma_\xi d\sigma_\eta\\
&\ \ \ \ +\int_{r(\xi)\geq \gamma}\int_{r(\eta)<\gamma}\frac{(\gamma-r(\eta))(r_2(\xi)-0)}{|\xi-\eta|^{n+2s}}d\sigma_\xi d\sigma_\eta\geq 0.
\end{split}\end{equation}
This proves \begin{equation}\begin{split}
&\int_{\mathbb{S}^n}\int_{\mathbb{S}^n}\frac{|\widetilde{r}(\xi)-\widetilde{r}(\eta)|^2}{|\xi-\eta|^{n+2s}}d\sigma_\xi d\sigma_\eta\\
&\ \ \geq \int_{\mathbb{S}^n}\int_{\mathbb{S}^n}\frac{|r_1(\xi)-r_1(\eta)|^2}{|\xi-\eta|^{n+2s}}d\sigma_\xi d\sigma_\eta+\int_{\mathbb{S}^n}\int_{\mathbb{S}^n}\frac{|r_2(\xi)-r_2(\eta)|^2}{|\xi-\eta|^{n+2s}}d\sigma_\xi d\sigma_\eta.
\end{split}\end{equation}
Then the proof of Lemma \ref{lem split} is accomplished.
\medskip

\bibliographystyle{amsalpha}

\begin{thebibliography}{10}

\bibitem{Au} T. Aubin, \textit{Probl\`emes isoperim\'etriques et espaces de Sobolev}. J. Differ. Geometry \textbf{11} (1976), 573-598.

\bibitem{BaWeWi} T. Bartsch, T. Weth and M. Willem, \textit{A Sobolev inequality with remainder term and critical equations on domains with topology for the polyharmonic operator}. Calc. Var. Partial
Differential Equations \textbf{18} (2003), 253-268.


\bibitem{Be1992} W. Beckner, {\it Sobolev inequalities, the Poisson semigroup, and analysis on the sphere $\mathbb{S}^n$}. Proc. Nat. Acad. Sci. U.S.A. \textbf{89} (1992), no. 11, 4816-4819.

\bibitem{Be1993} W. Beckner, \textit{Sharp Sobolev inequalities on the sphere and the Moser-Trudinger inequality}. Ann. of Math. (2) \textbf{138} (1993), no. 1, 213-242.


\bibitem{Be1997} W. Beckner, \textit{Logarithmic Sobolev inequalities and the existence of singular integrals}. Forum Math. (3) \textbf{9} (1997), 303-323.

\bibitem{Be2015} W. Beckner, \textit{Functionals for multilinear fractional embedding}. Acta Math. Sin. (Engl. Ser.) \textbf{31} (2015), no. 1, 1-28.

\bibitem{BiEg} G. Bianchi and H. Egnell, \textit{A note on the Sobolev inequality}. J. Funct. Anal. \textbf{100} (1991), no. 1, 18-24.

\bibitem{BDNS} M. Bonforte, J. Dolbeault, B. Nazaret and Nikita, \textit{Stability in Gagliardo-Nirenberg-Sobolev inequalities, Flows, regualrity and the entroy method}. arXiv:2007.03674v2.

\bibitem{BrLi} H. Brezis and E. Lieb, \textit{Sobolev inequalities with remainder terms}. J. Funct. Anal. \textbf{62} (1985),73-86.


\bibitem{Carlen} E. Carlen, \textit{Duality and stability for functional inequalities}. Ann. Fac. Sci. Toulouse Math. \textbf{26} (2017), 319-350.

\bibitem{CaF} E. Carlen and A. Figalli, \textit{Stability for a GNS inequality and the log-HLS inequality, with application to the critical mass Keller-Segel equation}. Duke Math. J. \textbf{162} (2013), no. 3, 579-625.

\bibitem{CFLL} C. Cazacu, J. Flynn, N. Lam and G. Lu, \textit{Caffarelli-Kohn-Nirenberg identities, inequalities and their stabilities},   J. Math. Pures Appl. (9) 182 (2024), 253-284.

\bibitem{ChFrWe} S. Chen, R. Frank and T. Weth, \textit{Remainder terms in the fractional Sobolev inequality}. Indiana Univ. Math. J. \textbf{62} (2013), no. 4, 1381--1397.

\bibitem{CLT} L. Chen, G. Lu and H. Tang, {\it
Sharp Stability of log-Sobolev and Moser-Onofri inequalities on the Sphere}, arXiv:2210.06727, J. Funct. Anal. (2023), https://doi.org/10.1016/j.jfa.2023.110022.

\bibitem{CLT1} L. Chen, G. Lu and H. Tang, {\it
Stability of Hardy-Littlewood-Sobolev inequalities with explicit lower bounds}, Adv. Math. \textbf{450} (2024), https://doi.org/10.1016/j.aim.2024.109778.

\bibitem{CLT2} L. Chen, G. Lu, H. Tang, {\it Optimal stability of Hardy-Littlewood-Sobolev and Sobolev inequalities
of arbitrary orders with dimension-dependent constants}, arXiv:2405.17727v1.

\bibitem{CFMP} A. Cianchi, N. Fusco, F. Maggi and A. Pratelli, \textit{The sharp Sobolev inequality in quantitative form}. J. Eur. Math. Soc.(JEMS) \textbf{11} (2009), no. 5, 1105-1139.

    \bibitem{DFLL}A. Dao, J. Flynn,  N. Lam and G. Lu,
\textit{$L^p$-Caffarelli-Kohn-Nirenberg inequalities and their stabilities}, arXiv:2310.07083.



\bibitem{DEFFL} J. Dolbeault, M. Esteban, A. Figalli, R. Frank and M. Loss, \textit{Stability for the Sobolev inequality with explicit constants}. arXiv:2209.08651, 2022.

\bibitem{FG} C. Fefferman and C. R. Graham, \textit{The ambient metric. Annals of Mathematics Studies, 178. Princeton University Press, Princeton}. NJ, 2012.

\bibitem{FiNe} A. Figalli and R. Neumayer, \textit{Gradient stability for the Sobolev inequality: the case $p\geq2$}. J. Eur. Math. Soc. (JEMS) \textbf{21} (2019), no. 2, 319-354.

\bibitem{FiZh} A. Figalli and Y. Zhang, \textit{Sharp gradient stability for the Sobolev inequality }.  Duke Math. J. \textbf{171} (2022), no. 12, 2407-2459.

\bibitem{JF1}  A. Figalli and D. Jerison, \textit{Quantitative stability for the Brunn-Minkowski inequality.} Adv. Math. 314 (2017), 1-47.

\bibitem{JF2}  A. Figalli and D. Jerison, \textit{Quantitative stability for sumsets in $R^n$.} J. Eur. Math. Soc. (JEMS) 17 (2015), no. 5, 1079-1106.
\bibitem{FrLi2012} R.~L.~Frank and E.~H.~Lieb, \textit{A new, rearrangement-free proof of the sharp Hardy-Littlewood-Sobolev inequality}. Spectral theory, function spaces and inequalities, 55-67, Oper. Theory Adv. Appl. \textbf{219}, Birkhauser/Springer Basel AG, Basel, 2012.

\bibitem{FKT} R. Frank, T. K\"{o}nig, H. Tang, {\it Classification of solutions of an equation related to a conformal log
Sobolev inequality}, Adv. Math. \textbf{375} (2020), 27pp.

\bibitem{GrJeMaSp} C. Graham, R. Jenne, L. Mason, G. Sparling, \textit{Conformally invariant powers of the Laplacian. I. Existence}. J. London Math. Soc. (2) \textbf{46} (1992), no. 3, 557-565.

\bibitem{Ko} T. K\"{o}nig, \textit{On the sharp constant in the Bianchi-Engell stability inequality}.  Bull. Lond. Math. Soc. 55 (2023), no. 4, 2070-2075.

\bibitem{Ko1} T. K\"{o}nig, \textit{Stability for the Sobolev inequality: existence of a minimizer}, arXiv: 2211.14185, to appear in JEMS.

\bibitem{Li} E. Lieb, \textit{Sharp constants in the Hardy--Littlewood--Sobolev and related inequalities}. Ann. of Math. (2) \textbf{118} (1983), no. 2, 349-374.

\bibitem{Lions1} P. L.  Lions,   The concentration-compactness principle in the calculus of variations. The limit case. I. Rev. Mat. Iberoamericana 1 (1985), no. 1, 145-201.

\bibitem{Lions2} P. L. Lions,  The concentration-compactness principle in the calculus of variations. The limit case. II. Rev. Mat. Iberoamericana 1 (1985), no. 2, 45-121.





\bibitem{LuWe} G. Lu and J. Wei, \textit{ On a Sobolev inequality with remainder terms}. Proc. Amer. Math. Soc. \textbf{128}
(1999), 75-84.

\bibitem{M} C. M$\ddot{u}$ller, \textit{Spherical Harmonics.} Lecture Notes in Mathematics, Vol. 17, Springer Verlog, Berlin, 1966.

\bibitem{Ro} G.~Rosen, \textit{Minimum value for $c$ in the Sobolev inequality $\|\phi^3\|\leq c\|\nabla\phi\|^3$}. SIAM J. Appl. Math. \textbf{21} (1971), 30-32.

\bibitem{SteinWeiss} E. Stein, G. Weiss, \textit{Introduction to Fourier analysis on Euclidean spaces}. Princeton Mathematical Series, No. 32. Princeton University Press, Princeton, N.J., 1971

\bibitem{Ta} G. Talenti, \textit{Best constants in Sobolev inequality}. Ann. Mat. Pura Appl. \textbf{110} (1976), 353-372.

\bibitem{WW} Z. Wang and M. Willem, \textit{Caffarelli-Kohn-Nirenberg inequalities with
remainder terms}. J. Funct. Anal. \textbf{203} (2003), 550-568.





\end{thebibliography}

\end{document}